\documentclass[12pt]{amsart}

\usepackage[latin1]{inputenc}
\usepackage{amssymb,amsmath,amsfonts, amsthm}
\usepackage{enumerate}
\usepackage{mathrsfs}
\usepackage{times}
\usepackage{xcolor}
\usepackage{cite}

\newtheorem{theorem}{Theorem}[section]
\newtheorem{lemma}[theorem]{Lemma}
\newtheorem{proposition}[theorem]{Proposition}
\newtheorem{corollary}[theorem]{Corollary}
\numberwithin{equation}{section}

\newcommand{\fconvse}{{\mbox{\rm Conv}_{{\rm sc}}(\R)}} 
\newcommand{\fconvs}{{\mbox{\rm Conv}_{{\rm sc}}(\R^n)}} 
\newcommand{\fconvsk}{{\mbox{\rm Conv}_{{\rm sc}}(\R^k)}} 
\newcommand{\fconvskpt}{{\mbox{\rm Conv}_{{\rm sc}}(\R^{k+2})}} 
\newcommand{\fconvsz}{{\mbox{\rm Conv}_{{\rm z}}(\R^{n})}} 
\newcommand{\fconvsg}{{\mbox{\rm Conv}_{{\rm gz}}(\R^{n})}} 
\newcommand{\fconvx}{{\mbox{\rm Conv}(\R^n)}}
\newcommand{\fconvf}{{\mbox{\rm Conv}(\R^n; \R)}}
\newcommand{\fconvfk}{{\mbox{\rm Conv}(\R^k; \R)}}
\newcommand{\fconvfE}{{\mbox{\rm Conv}(E; \R)}}
\newcommand{\fconvfF}{{\mbox{\rm Conv}(F; \R)}}
\newcommand{\fconvsE}{{\mbox{\rm Conv}_{{\rm sc}}(E)}} 
\newcommand{\fconvsF}{{\mbox{\rm Conv}_{{\rm sc}}(F)}} 
\newcommand{\fconvfs}{{\mbox{\rm Conv}_{0}(\R^n;\R)}}
\newcommand{\fconvfsk}{{\mbox{\rm Conv}_{0}(\R^k;\R)}}

\newcommand{\infconv}{\mathbin{\Box}}
\newcommand{\deq}{\mathbin{:=}}
\newcommand{\sq}{\mathbin{\vcenter{\hbox{\rule{.3ex}{.3ex}}}}}
\newcommand{\proj}{\operatorname{proj}}

\DeclareMathOperator{\acosh}{acosh}
\newcommand{\Abel}{\operatorname{\mathcal A}} 
\newcommand{\iAbel}{\operatorname{\mathcal A}^{-1}\!} 

\newcommand{\Grass}[2]{\operatorname{G}(#2,#1)}

\newcommand{\cK}{{\mathcal K}} 
\newcommand{\cP}{{\mathcal P}}

\newcommand{\Val}[1]{\operatorname{Val}(\R^{#1})}
\newcommand{\Vals}[1]{\operatorname{Val}^{\rm sm}(\R^{#1})}
\newcommand{\VConvf}{\operatorname{VConv}(\R^n; \R)}

\renewcommand{\S}{\mathbb{S}}
\newcommand{\sn}{{\mathbb{S}^{n-1}}}
\newcommand{\Bn}{B^n}

\newcommand{\gln}{\operatorname{GL}(n)}

\renewcommand{\O}{\operatorname{O}}
\newcommand{\SO}{\operatorname{SO}}

\newcommand{\Hess}{{\operatorname{D}}^2}

\DeclareMathOperator{\oZ}{\operatorname{Z}}
\DeclareMathOperator{\oY}{\operatorname{Y}}
\newcommand{\oZZ}[2]{\operatorname{V}_{#1,#2}}

\newcommand{\cT}{\mathcal{T}}

\newcommand{\R}{{\mathbb R}}
\newcommand{\interno}{\operatorname{int}}
\newcommand{\dom}{\operatorname{dom}}
\newcommand{\hm}{\mathcal H}
\newcommand{\epi}{\operatorname{epi}}
\renewcommand{\d}{\,\mathrm{d}}
\newcommand{\ind}{{\rm\bf I}}

\newcommand{\Had}[2]{D_{#1}^{#2}} 

\begin{document}
\title[The Hadwiger Theorem on Convex Functions, IV]{The Hadwiger Theorem on Convex Functions, IV:\\ The Klain Approach}

\author{Andrea Colesanti}
\address{Dipartimento di Matematica e Informatica ``U. Dini'',
Universit\`a degli Studi di Firenze,
Viale Morgagni 67/A - 50134, Firenze, Italy}
\email{andrea.colesanti@unifi.it}

\author{Monika Ludwig}
\address{Institut f\"ur Diskrete Mathematik und Geometrie,
Technische Universit\"at Wien,
Wiedner Hauptstra\ss e 8-10/1046,
1040 Wien, Austria}
\email{monika.ludwig@tuwien.ac.at}

\author{Fabian Mussnig}
\address{Dipartimento di Matematica e Informatica ``U. Dini'',
Universit\`a degli Studi di Firenze,
Viale Morgagni 67/A - 50134, Firenze, Italy}
\email{mussnig@gmail.com}

\date{}

\begin{abstract} New proofs of the Hadwiger theorem for smooth and for continuous valuations on convex functions are obtained, and the Klain--Schneider theorem on convex functions is established. In addition,  an extension theorem for valuations defined on functions with lower dimensional domains is proved, and its connection to the Abel transform is explained.

\bigskip
{\noindent 2000 AMS subject classification: 52B45 (26B25, 49Q20, 52A41, 52A39)}
\end{abstract}

\maketitle

\section{Introduction and Statement of Results}
On a space $X$ of (extended) real-valued functions, a functional $\oZ\colon X\to\R$ is called a \emph{valuation} if
$$\oZ(u)+\oZ(v)=\oZ(u\vee v) + \oZ(u \wedge v)$$
for every $u,v\in X$ such that also their pointwise maximum $u\vee v$ and their pointwise minimum $u\wedge v$ belong to $X$. This recently introduced notion extends the classical notion of valuation on the space of convex bodies (that is, non-empty, compact, convex sets) in $\R^n$.  On convex bodies, the paradigmatic result on valuations is the celebrated Hadwiger theorem \cite{Hadwiger:V}, which provides a complete classification of continuous, translation and rotation invariant valuations on convex bodies and a characterization of the classical intrinsic volumes. Hadwiger's theorem leads to effortless proofs  in integral geometry and geometric probability (see \cite{Hadwiger:V} and \cite{Klain:Rota}). It is the first culmination of the program initiated by Blaschke to classify valuations invariant under various transformation groups and the starting point of geometric valuation theory (see \cite[Chapter~6]{Schneider:CB2} and see \cite{Alesker99,Alesker01,AleskerFaifman, Bernig:Fu,  LiMa,  Ludwig:matrix, Ludwig:Reitzner2,Ludwig:convex,Haberl:Parapatits_centro, HLYZ_acta, BLYZ_cpam} for some recent results).
The first classification results of valuations on classical function spaces were obtained for $L_p$ and Sobolev spaces and  on continuous and   Lipschitz functions (see  \cite{Tsang:Lp, Ludwig:SobVal,Ludwig:Fisher,ColesantiPagniniTradaceteVillanueva,ColesantiPagniniTradaceteVillanueva2021, Villanueva2016}).  

For valuations on convex functions,  the first classification results \cite{Colesanti-Ludwig-Mussnig-1,Colesanti-Ludwig-Mussnig-2,Mussnig19, Mussnig21} and the first structural results \cite{Alesker_cf,Colesanti-Ludwig-Mussnig-4, Knoerr1, Knoerr2} were recently obtained. In \cite{Colesanti-Ludwig-Mussnig-5}, the authors  established  the following Hadwiger theorem for convex functions. Let
$$\fconvs:=\Big\{u\colon\R^n\to(-\infty,+\infty]\colon u \not\equiv +\infty, \lim_{|x|\to+\infty} \frac{u(x)}{|x|}=+\infty, u \text{ is l.s.c. and convex}\Big\}$$
denote the space of proper,  super-coercive, lower semicontinuous, convex functions on $\R^n$, where $\vert\cdot\vert$ is the Euclidean norm. It is equipped with the topology induced by epi-convergence (see Section \ref{convex functions} for the definition). A functional $\oZ\colon\fconvs\to\R$ is \emph{epi-translation invariant} if $\oZ(u\circ \tau^{-1}+\alpha)=\oZ(u)$ for every $u\in\fconvs$, translation $\tau$ on $\R^n$ and  $\alpha\in\R$. It is \emph{rotation invariant} if $\oZ(u\circ \vartheta^{-1})=\oZ(u)$ for every $u\in\fconvs$ and $\vartheta\in\SO(n)$. Let $n\ge 2$.

\begin{theorem}[\!\cite{Colesanti-Ludwig-Mussnig-5}]
\label{thm:hadwiger_convex_functions}
A functional $\oZ\colon\fconvs \to \R$ is a continuous, epi-translation and rotation invariant valuation if and only if there exist functions $\zeta_0\in\Had{0}{n}$, \dots, $\zeta_n\in\Had{n}{n}$  such that
\begin{equation*}
\oZ(u)= \sum_{j=0}^n \oZZ{j}{\zeta_j}^n(u) 
\end{equation*}
for every $u\in\fconvs$.
\end{theorem}

\noindent
Here, the spaces $\Had{j}{n}$ are defined for $0\leq j \leq n-1$ as
$$
\Had{j}{n}\deq\Big\{\zeta\in C_b((0,\infty))\colon  \lim_{s\to 0^+} s^{n-j} \zeta(s)=0,  \lim_{s\to 0^+} \int_s^{\infty}  t^{n-j-1}\zeta(t) \d t \text{ exists and is finite}\Big\},
$$
where  $C_b((0,\infty))$ is the set of continuous functions with bounded support on $(0,\infty)$. In addition, let $\zeta\in\Had{n}{n}\,$ if $\zeta\in C_b((0,\infty))$ and $\lim_{s\to 0^+} \zeta(s)$ exists and is finite. In this case, we set $\zeta(0)\deq\lim_{s\to 0^+} \zeta(s)$ and consider $\zeta$ also as an element of $C_c([0,\infty))$, the set of continuous functions with compact support on $[0,\infty)$. For $0\leq j \leq n$ and $\zeta\in\Had{j}{n}$, the functionals $\oZZ{j}{\zeta}^n\colon\fconvx\to\R$ were introduced in \cite{Colesanti-Ludwig-Mussnig-5}, where it was proved that there exists a unique continuous extension to $\fconvs$ of the functional defined by
\begin{equation}
\label{eq:ozz_on_c2p}
u\mapsto\int_{\R^n} \zeta(|\nabla u(x)|)\big[\Hess u(x)\big]_{n-j} \d x
\end{equation}
on $\fconvs\cap C_+^2(\R^n)$. Here,  $C_+^2(\R^n)$ is the class of twice continuously differentiable functions $u$ with positive definite Hessian $\Hess u$ and  $\big[\Hess u]_k$ is the $k$th elementary symmetric function of the eigen\-values of $\Hess u$ for $1\le k\le n$, while $[\Hess u]_0\!:\equiv 1$. This extension is called the \emph{$j$th functional intrinsic volume} $\,\oZZ{j}{\zeta}^n$ with density $\zeta$.
Explicit representations for functional intrinsic volumes were obtained in \cite{Colesanti-Ludwig-Mussnig-6} using functional Cauchy--Kubota formulas (see Theorem \ref{cauchy_function} below) and in \cite{Colesanti-Ludwig-Mussnig-7} using a new family of mixed Monge--Amp\`ere measures.
Note that $\oZZ{0}{\zeta}^n(u)$ is a constant for $\zeta\in\Had{0}{n}$, independent of $u\in\fconvs$.
As in the classical Hadwiger theorem, a complete classification of continuous,  epi-translation and rotation invariant valuations is obtained and thereby a characterization of functional intrinsic volumes. 

The main aim of this paper is to give a new proof of Theorem~\ref{thm:hadwiger_convex_functions}. The proof in \cite{Colesanti-Ludwig-Mussnig-5} followed the basic outline of Hadwiger's original proof. Klain \cite{Klain95} found a different approach to the classical Hadwiger theorem, which we try to adapt to the functional case. A critical element of Klain's proof is his so-called volume theorem, which was extended by Schneider \cite{Schneider:simple} to a complete classification of  \emph{simple}, continuous and translation invariant valuations on convex bodies, where a valuation is simple if it vanishes on lower dimensional sets. We establish the following  new functional version of the Klain--Schneider theorem. 

\begin{theorem}
\label{Klain-Schneider}
A functional $\,\oZ\colon\fconvs \to \R$ is a simple, continuous, epi-translation invariant valuation if and only if there exists a function 
$\zeta\in C_c(\R^n)$  such that
\begin{equation*}
\oZ(u)= \int_{\dom u} \zeta(\nabla u(x)) \d x
\end{equation*}
for every $u\in\fconvs$.
\end{theorem}

\noindent
Here a valuation on $\fconvs$ is called \emph{simple} if it vanishes on functions with lower dimensional domains and the domain of $u\in\fconvs$ is defined as $\dom u\deq\{x\in\R^n\colon u(x)<\infty\}$.
The proof of Theorem~\ref{Klain-Schneider} is given in Section~\ref{se:simple_vals}, where we also show that functional analogs of generalized zonoids are dense in $\fconvs$ (see Lemma~\ref{convex_approx} below).

The second critical element in Klain's proof is an extension of valuations defined on lower dimensional sets to valuations defined on general convex bodies. While such an extension is straightforward on convex bodies, the dependence of functional intrinsic volumes on the classes $\Had{j}{n}$ constitutes an obstacle in the functional setting.
In Section \ref{se:smooth_vals}, we establish such an extension when we restrict to so-called smooth valuations (see Section \ref{se:smooth_vals} for the definition) and thereby give a new proof of a result recently established by Knoerr \cite{Knoerr3}, who transferred Alesker's notion of smooth valuation from convex bodies to convex functions.
A version of Knoerr's result is stated next.
\goodbreak
\pagebreak

\begin{theorem}[\!\cite{Knoerr3}]
\label{hadwiger_smooth}
A functional $\,\oZ\colon\fconvs \to \R$ is a smooth, epi-translation and rotation invariant valuation if and only if there exist  functions 
$\varphi_0,\dots, \varphi_n\in C_c^\infty([0,\infty))$  such that
\begin{equation*}
\oZ(u)=  \sum_{j=0}^n\int_{\R^n} \varphi_j(\vert y\vert^2) \d \Psi^n_j(u, y)
\end{equation*}
for every $u\in\fconvs$.
Moreover,
\begin{equation}
\label{eq:smooth_differentiable}
\oZ(u)=\sum_{j=0}^n \int_{\R^n} \varphi_j(\vert\nabla u(x)\vert^2)\big[\Hess u(x)\big]_{n-j} \d x
\end{equation}
for every $u\in\fconvs\cap C_+^2(\R^n)$.
\end{theorem}

\noindent Here, $\Psi^n_j(u,\cdot)$ are  Hessian measures associated to $u$ (see Section \ref{se:hessian} for the definition). It follows from the continuity of all functionals combined with \eqref{eq:ozz_on_c2p} and \eqref{eq:smooth_differentiable}  that the valuation $\oZ$ in Theorem~\ref{hadwiger_smooth} can be written as
$$\oZ=\sum_{j=0}^n \oZZ{j}{\zeta_j}^n$$
where $\zeta_j\in C_c^\infty([0,\infty))$ is given by $\zeta_j(t)\deq\varphi_j(t^2)$ for $t\geq 0$ and $0\leq j\leq n$.
We remark that Theorem~\ref{hadwiger_smooth} can be obtained as a consequence of Theorem \ref{thm:hadwiger_convex_functions}. However, neither Knoerr's proof nor our new proof of his result use Theorem \ref{thm:hadwiger_convex_functions}. 
\goodbreak

In Section~\ref{se:proof_hadwiger}, we  give a new proof of Theorem~\ref{thm:hadwiger_convex_functions} that makes use of Theorem \ref{hadwiger_smooth}. 
In fact, we prove an integral-geometric version of the Hadwiger theorem on convex functions, which was established by the authors in \cite{Colesanti-Ludwig-Mussnig-6} (see Theorem~\ref{thm2:hadwiger_convex_functions_ck} below) and which was shown there to be equivalent to Theorem~\ref{thm:hadwiger_convex_functions}. Functional Cauchy--Kubota formulas,  which were established in \cite{Colesanti-Ludwig-Mussnig-6} (see Theorem \ref{cauchy_function} below), are the essential tool in this new approach and in our new proof of the Hadwiger theorem on convex functions.

In Section \ref{se:lower}, we discuss valuations on functions with lower dimensional domains. In particular, we address the question when a functional defined on functions with lower dimensional domains can be extended to general convex functions and the connection of this question to the Abel transform. We also explain why Klain's proof of Hadwiger's theorem cannot be transferred immediately to  general continuous valuations in the functional case.

\section{Preliminaries}
We work in $n$-dimensional Euclidean space $\R^n$, with $n\ge 1$, endowed with the Euclidean norm $\vert \cdot\vert $ and the usual scalar product
$\langle \cdot,\cdot\rangle$. We also use coordinates, $x=(x_1,\dots,x_n)$, for $x\in\R^n$ and write $e_1, \dots, e_n$ for the vectors of the standard orthonormal basis. For $k\le n$, we often identify $\R^k$ with $
\{x\in\R^n\colon x_{k+1}=\dots=x_n=0\}$. Let $\Bn\deq\{x\in\R^n\colon \vert x\vert \le 1\}$ be the Euclidean unit ball and $\sn$ the unit sphere in $\R^n$. 
Write $\kappa_j$ for the $j$-dimensional volume of $B^j$ and set $\kappa_0\deq1$.

\subsection{Convex Bodies}\label{convex bodies} A basic reference on convex bodies is the book by Schneider \cite{Schneider:CB2}. Let $\cK^n$ denote the set of convex bodies in $\R^n$. For $K\in\cK^n$, its {\em support function} $h_K\colon\R^n\to \R$ is defined as
$$h_K(y)\deq \max\nolimits_{x\in K}\langle x,y \rangle.$$
It is a one-homogeneous and convex function that determines $K$. The topology on $\cK^n$ is induced by the Hausdorff distance which is defined for $K,L\in\cK^n$ as $\max_{y\in\sn} \vert h_K(y)-h_L(y)\vert$.

\subsection{Valuations on Convex Bodies}
We use \cite[Chapter~6]{Schneider:CB2} as a general reference. 
The following result is the Klain--Schneider theorem, which provides a complete classification of simple, continuous and translation invariant valuations.

\begin{theorem}[Klain \cite{Klain95}, Schneider \cite{Schneider:simple}]
\label{KlainSchneider}
A functional $\oZ\colon \cK^{n}\to \R$ is a simple, continuous, translation invariant valuation if and only if there exist $\gamma\in\R$  and an odd continuous function $\sigma\colon\sn\to \R$ such that
$$\oZ(K)=\gamma\,V_{n}(K)+ \int_{\sn} \sigma(y) \d S_{n-1}(K,y)$$
for every $K\in\cK^{n}$.
\end{theorem}

\noindent Here, $S_j(K,\cdot)$ is the $j$th area measure of $K\in\cK^n$  (see \cite[Section 2.5]{Schneider:CB2} for the definition). 

\goodbreak
In the proof of the functional version of the Klain--Schneider theorem (Theorem~\ref{Klain-Schneider}), we use the following  weaker version of the above result.

\begin{theorem}[Klain \cite{Klain95}]
\label{Klain}
If $\,\oZ\colon \cK^{n}\to \R$ is a simple, continuous, translation invariant valuation, then   there exists $\gamma\in\R$  such that
$$\oZ(K)=\gamma\,V_{n}(K)$$
for every $K\in\cK^{n}$ that is origin-symmetric.
\end{theorem}

\noindent
Here, a convex body $K\in\cK^{n}$ is called \emph{origin-symmetric} if $K=-K$ and $-K\deq\{-x\colon x\in K\}$ is the reflection of $K$ at the origin.
\goodbreak

Let $\Val{n}$ denote the space of continuous, translation invariant valuations on $\cK^n$. With the norm
$$\Vert \oZ \Vert \deq \sup\{\vert \oZ(K)\vert\colon K\in\cK^n, K\subseteq \Bn\},$$
the space $\Val{n}$ is a Banach space. A natural representation of  $\gln$ on $\Val{n}$ is defined by
$$\vartheta\mapsto \oZ\circ \vartheta^{-1}.$$
A valuation $\oZ \in \Val{n}$
is \emph{smooth} if the mapping $\vartheta \mapsto  \oZ\circ\vartheta^{-1}$ from $\gln$ into the Banach space $\Val{n}$ is infinitely differentiable. The subspace $\Vals{n}$ of smooth valuations is dense in $\Val{n}$. 
Note that for a linear subspace $E$ of $\R^n$, the restriction of a smooth valuation $\oZ\colon\cK^n\to\R$ to convex bodies in $E$ is again smooth.

Alesker \cite{Alesker2006} showed that valuations from $\Vals{n}$ are those which can be represented as integrals of smooth differential forms on the so-called normal cycle.  Combined with a result by Z\"ahle \cite{Zahle1986} on the representation of generalized curvature measures by differential forms, we obtain the following statement: For $\xi\in C(\sn)$ and 
$1\le j\le n-1$, the valuation $\oZ\colon\cK^n\to \R$, defined by
\begin{equation}\label{smooth_valuation}
\oZ(K):=\int_{\sn} \xi(y)\d S_j(K,y),
\end{equation}
is smooth if and only if $\xi\in C^{\infty}(\sn)$. Here, we also used that if the integral in \eqref{smooth_valuation} vanishes for a given $\xi\in C(\sn)$ for all $K\in\cK^n$, then $\xi$ is linear (and hence smooth) by \cite[Theorem 3.5]{Weil80}.

\goodbreak

\subsection{Convex Functions}\label{convex functions}
We collect some basic results and properties of convex functions. For standard references on this subject, we refer to the books by Rockafellar \cite{Rockafellar} and Rockafellar \& Wets \cite{RockafellarWets}.

Let $\fconvx$ be the set of proper, lower semicontinuous, convex functions $u\colon\R^n\to(-\infty,\infty]$. 
Every function $u\in\fconvx$ is uniquely determined by its \emph{epi-graph}
$$\epi u:= \{(x,t)\in \R^n\times \R\colon u(x)\leq t\}$$
which is a closed, convex subset of $\R^{n+1}$. 
For $t\in\R$, we write
$$\{u\leq t\}:=\{x\in\R^n\colon u(x)\leq t\}$$
for its \emph{sublevel sets}, which are closed, convex subsets of $\R^n$, as $u$ is lower semicontinuous. 
\goodbreak

The standard topology on $\fconvx$ and its subsets is induced by epi-convergence. A sequence of functions $u_k\in\fconvx$ is \emph{epi-convergent} to $u\in\fconvx$ if for every $x\in\R^n$:
\begin{enumerate}
    \item[(i)] $u(x)\leq \liminf_{k\to\infty} u_k(x_k)$ for every sequence $x_k\in\R^n$ that converges to $x$;
    \item[(ii)] $u(x)=\lim_{k\to\infty} u_k(x_k)$ for at least one sequence $x_k\in\R^n$ that converges to $x$.
\end{enumerate}
Note that the limit of an epi-convergent sequence of extended real-valued functions on $\R^n$ is always lower semicontinuous.

Let $\fconvf$ denote the set of finite-valued convex functions in $\fconvx$. A sequence of functions $v_k\in\fconvf$ is epi-convergent to $v\in\fconvf$ if and only if $v_k$ converges pointwise to $v$, which by convexity is equivalent to uniform convergence on compact sets. 

For $u\in\fconvx$, we denote by $u^*\in\fconvx$ its \emph{Legendre--Fenchel transform} or \emph{convex conjugate}, which is defined by
$$u^*(y):=\sup\nolimits_{x\in\R^n} \big(\langle x,y \rangle - u(x) \big)$$
for $y\in\R^n$. Since $u$ is lower semicontinuous, $u^{**}=u$. Moreover, $u\in\fconvs$ if and only if $u^*\in\fconvf$ and $u\in \fconvs\cap C_+^2(\R^n)$ if and only if $u^*\in \fconvs \cap C_+^2(\R^n)$. 
In addition, a sequence of functions $u_k\in\fconvx$ is epi-convergent to $u\in\fconvx$ if and only if $u_k^*$ is epi-convergent to $u^*$.

Let $\Grass{k}{n}$ be the Grassmannian of $k$-dimensional subspaces in $\R^n$, where $0\le k\le n$. 
For a subspace $E\in\Grass{k}{n}$, let $\fconvsE$ denote the set of functions $u\colon E\to (-\infty, \infty]$ that are proper, lower semicontinuous, super-coercive, and convex.
We identify $\fconvsE$ with the space
\begin{equation*}
\{u\in\fconvs\colon \dom u \subseteq E\},
\end{equation*}
which gives a continuous embedding of $\fconvsE$ into $\fconvs$. For $E\in\Grass{k}{n}$ and a  function $u\in\fconvs$, we define the \emph{projection function} $\proj_E u\colon E\to (-\infty,\infty]$ by
$$\proj_E u(x_E):= \min\nolimits_{z\in E^\perp} u(x_E+z)$$
where $x_E\in E$ and  $E^\perp$ is the orthogonal complement of $E$. Since $\min_{z\in E^\perp} u(x_E+z) \leq t$ if and only if there exists $z\in E^\perp$ such that $u(x_E+z)\leq t$, we have
\begin{equation}
\label{eq:proj_fct_lvl_set}
\{\proj_E u \leq t\} = \proj_E \{u\leq t\}
\end{equation}
for every $t\in\R$, where $\proj_E$ on the right side of \eqref{eq:proj_fct_lvl_set} denotes the usual orthogonal projection of sets onto $E$. In addition,
$\epi (\proj_E u) = \proj_{E\times \R} (\epi u)$.
In particular, it is clear that $\proj_E u \in \fconvsE$. If $u\in\fconvsE$ is differentiable, then $\nabla_E\, u$ denotes the gradient of $u$ taken with respect to the ambient space $E$. 

\goodbreak
For $u\in\fconvx$, the \emph{subdifferential} of $u$ at $x\in\R^n$ is defined by
$$\partial u(x) := \{y\in\R^n\colon u(z)\geq u(x)+\langle y, z-x\rangle \text{ for } z\in\R^n \}.$$
If $u$ is differentiable at $x$, then $\partial u(x)=\{\nabla u(x)\}$. Note that a convex function is differentiable almost everywhere on the interior of its domain. For $x,y\in\R^n$, we have $y\in\partial u(x)$ if and only if $x\in\partial u^*(y)$.

\begin{lemma}[\!\cite{Colesanti-Ludwig-Mussnig-6}, Lemma 3.1]
\label{le:proj_subgradient}
Let $E\subseteq \R^n$ be a linear subspace and $u\in\fconvs$. If $x_E,y_E\in E$ are such that $y_E\in\partial \proj_E u(x_E)$, then for every $x\in\R^n$ such that $\proj_E u(x_E)=u(x)$ we also have $y_E\in\partial u(x)$. In particular, such $x\in\R^n$ exist and  $\proj_E x = x_E$.
\end{lemma}

For two functions $u_1,u_2\in\fconvs$ we denote by $u_1\infconv u_2\in\fconvs$ their \emph{epi-sum} or  \emph{infimal convolution} which is defined as
$$(u_1\infconv u_2)(x):=\inf\nolimits_{x_1+x_2=x} u_1(x_1)+u_2(x_2)$$
for $x\in\R^n$. Note that
$$\epi (u_1\infconv u_2)=\epi u_1 + \epi u_2,$$
where   Minkowski addition of subsets in $\R^{n+1}$ is used on the right side. We define \emph{epi-multiplication} by $\lambda>0$ on $\fconvs$ in the following way. For $u\in\fconvs$ and $x\in\R^n$, set
$$\lambda\sq u(x):=\lambda\, u\left( \frac x\lambda \right).$$
This corresponds to rescaling the epi-graph of $u$ by the factor $\lambda$, that is, $\epi (\lambda\sq u)=\lambda \epi u$.

\subsection{The Abel Transform}
\label{se:abel}
For $\zeta\in C_b((0,\infty))$, the  \emph{Abel transform} of $\zeta$ is defined as
\begin{equation*}
\Abel\zeta(t):=\int_{-\infty}^{\infty} \zeta(\sqrt{s^2+t^2}) \d s
\end{equation*}
for $t\in (0,\infty)$. Using the substitution $\sqrt{s^2+t^2}=t \cosh(r)$,
we may rewrite this as
\begin{equation}
\label{eq:def_abel_trans2}
\Abel\zeta(t)=2 \int_0^{\infty} \zeta(\sqrt{s^2+t^2}) \d s = 2 t \int_0^{\infty} \zeta(t \cosh(r)) \cosh(r) \d r.
\end{equation}
For $k\ge 1$ and $\zeta\in C_b((0,\infty))$, let $\Abel^k \zeta \deq (\Abel \circ\cdots\circ \Abel) \zeta$ where $\Abel$ is applied $k$ times.
Consequently,
\begin{equation}
\label{eq:abel_k}
\Abel^k \zeta(t)=\int_{\R^k} \zeta(\sqrt{|x|^2+t^2}) \d x
\end{equation}
for every $t>0$.
The following property is well-known.

\goodbreak
\begin{lemma}
\label{le:abel_sq_diffable}
For $\zeta\in C_b((0,\infty))$, we have
\begin{equation}\label{eq:abel_twice}
\Abel^2 \zeta(t) = 2 \pi \int_t^{\infty} \zeta(s) s \d s
\end{equation}
for every $t>0$, and, in particular, $\Abel^2 \zeta\in C_b^1((0,\infty))$.
\end{lemma}
\begin{proof}
Using \eqref{eq:abel_k} and polar coordinates, we obtain
$$\Abel^2 \zeta(t) = \int_{\R^2} \zeta(\sqrt{|x|^2+t^2})\d x = 2\pi \int_0^{\infty} \zeta(\sqrt{r^2+t^2})\,r \d r$$
for every $t\in(0,\infty)$, which combined with the substitution $s=\sqrt{r^2+t^2}$ gives \eqref{eq:abel_twice}.
It is now easy to see that $\Abel^2\zeta$ has continuous derivative and bounded support.
\end{proof}

For $\zeta\in C_b^1((0,\infty))$, the inverse Abel transform is given by $$\iAbel\zeta(s)=-\frac{1}{\pi} \int_s^{\infty} \frac{\zeta'(t)}{\sqrt{t^2-s^2}} \d t$$ for $s\in (0,\infty)$. Using the substitution $t = s \cosh(r)$, this can be written as $$\iAbel\zeta(s)=-\frac{1}{\pi} \int_0^\infty \zeta'(s \cosh(r)) \d r.$$ Note that if $\zeta\in C_c^\infty([0,\infty))$, then also $\iAbel \zeta\in C_c^\infty([0,\infty)).$

\subsection{Hessian Measures}\label{se:hessian}
We recall the definition of two families of measures used in our results. They are both marginals of more general Hessian measures, see \cite{ColesantiHug2000,Colesanti-Ludwig-Mussnig-3}. We remark that one of these families of Hessian measures was introduced by Trudinger and Wang \cite{Trudinger:Wang1997, Trudinger:Wang1999}.

For $u\in\fconvs\cap C_+^2(\R^n)$ and $0\leq j \leq n$, there exists a non-negative Borel measure $\Psi_j^n(u,\cdot)$ such that
\begin{equation}
\label{eq:def_psi_j_n}
\int_{\R^n} \beta(y) \d\Psi^n_j(u,y)=\int_{\R^n} \beta(\nabla u(x)) \big[\Hess u(x)\big]_{n-j}\d x
\end{equation}
for every Borel function $\beta\colon\R^n\to[0,\infty)$. The \emph{Hessian measures} $\Psi_j^n(u,\cdot)$, as functions of $u$, are continuous (with respect to the topology induced by weak convergence of measures) and continuously extend to $\fconvs$. That is, if a sequence $u_k$ in $\fconvs$  epi-converges to $u\in\fconvs$, then
 \begin{equation*}\label{weak convergence}
\lim_{k\to+\infty}\int_{\R^n}\beta(y)\d\Psi_j^n(u_k,y)=\int_{\R^n}\beta(y)\d\Psi_j^n(u,y)
\end{equation*}
for every $\beta\in C_c(\R^n)$. Since $\fconvs\cap C_+^2(\R^n)$ is dense in $\fconvs$, it follows that \eqref{eq:def_psi_j_n} uniquely determines the measure $\Psi_j^n(u,\cdot)$. For the special case $j=n$, we have
\begin{equation}\label{n-gradient}
\int_{\R^n} \beta(y) \d\Psi^n_n(u,y)= \int_{\dom u} \beta(\nabla u(x))\d x
\end{equation}
for every $u\in\fconvs$ and $\beta\in C_c(\R^n)$.
Note that for $0\le j\le n$ the map $u\mapsto \Psi^n_j(u,\cdot)$ is a valuation on $\fconvs$, which is continuous, epi-translation invariant and rotation covariant.
Here, we say that $\Psi^n_j$ is \emph{rotation covariant} if $\Psi^n_j(u\circ \vartheta^{-1}, B)= \Psi^n_j(u,\vartheta^{-1} B)$ for every $u\in\fconvs$, every $\vartheta\in\SO(n)$ and every Borel set $B\subset \R^n$. So, in particular, we easily obtain the following results, where we use \eqref{n-gradient} in the first lemma.

\begin{lemma}[\!\cite{Colesanti-Ludwig-Mussnig-3}]  
	\label{gradient integral}
    For $\zeta\in C_c(\R^n)$, the functional $\oZ\colon\fconvs\to\R$, defined by
    $$
    \oZ(u)=\int_{\dom u}\zeta(\nabla u(x))\d x,
    $$
    is a continuous, epi-translation invariant valuation.
\end{lemma}

\begin{lemma}[\!\cite{Colesanti-Ludwig-Mussnig-4}, Theorem 1.2]  
	\label{fine zeta}
    For $\zeta\in C_c([0,\infty))$, the functional $\oZ\colon\fconvs\to\R$, defined by
    $$
    \oZ(u)=\int_{\R^n}\zeta(|y|)\d\Psi^n_j(u,y),
    $$
    is a continuous, epi-translation and rotation invariant valuation for every $0\leq j \leq n$.
\end{lemma}

For $v\in\fconvf\cap C_+^2(\R^n)$ and $0\leq j \leq n$, let $\Phi_j^n(v,\cdot)$ be the non-negative Borel measure that has the property that
\begin{equation}
\label{eq:hess_def}
\int_{\R^n} \beta(x) \d\Phi^n_j(v,x)=\int_{\R^n} \beta(x) \big[\Hess v(x)\big]_{j}\d x
\end{equation}
for every Borel function $\beta\colon\R^n\to[0,\infty)$. By continuity the measure $\Phi_j^n(v,\cdot)$ extends to $v\in\fconvf$. As a consequence of \cite[Theorem 8.2]{Colesanti-Ludwig-Mussnig-3}, we have
\begin{equation}
\label{eq:int_u_psi_int_v_phi}
\Psi_j^n(u,\cdot)=\Phi_j^n(u^*,\cdot)
\end{equation}
for every $u\in\fconvs$ and $0\leq j\leq n$. The measure $\Phi_n^n(v,\cdot)$ is the Monge--Amp\`ere measure of $v$.

\goodbreak
For orthogonal and complementary subspaces $E$ and $F$ of $\R^n$ and $x\in\R^n$, we write $x=(x_E,x_F)$ with $x_E\in E$ and $x_F\in F$. Let $\fconvfE\deq\{v_E\colon E\to \R\colon v_E \text{ is convex}\}$ and define $\fconvfF$ accordingly. For $v_E\in\fconvfE$ and $v_F\in\fconvfF$, we define $v_E+v_F\in\fconvf$ as
$$(v_E+v_F)(x):=v_E(x_E)+v_F(x_F)$$
for every $x=(x_E,x_F)\in\R^n$.

\begin{lemma}[\!\cite{Colesanti-Ludwig-Mussnig-5}, Lemma 4.8]
\label{le:decompose_hessian_measure}
Let $E$ and $F$ be orthogonal and complementary subspaces of $\R^n$ such that $\dim E=k$ with $1\le k<n$. If $v_E\in\fconvfE$ and $v_F\in\fconvfF$, then 
$$\Phi_l^n(v_E+v_F,B)=\sum_{i=0\vee(k+l-n)}^{k\wedge l} \Phi_i^k(v_E,B\cap E)\, \Phi_{l-i}^{n-k}(v_F,B\cap F)$$
for every $0\leq l \leq n$ and every Borel set $B\subseteq \R^n$.
\end{lemma}

\goodbreak
The next result follows from \cite[Lemma 2.4]{Alesker_cf}. See also (4.10) and (4.11) in \cite{Colesanti-Ludwig-Mussnig-5}.
\begin{lemma}
\label{le:delta}
Let $\bar{x}_1,\dots,\bar{x}_n\in\R\backslash\{0\}$ and set $\bar{x}\deq(\bar{x}_1,\ldots,\bar{x}_n)$. For $\bar{v}\in\fconvf$, defined as
$$\bar{v}(x_1,\ldots,x_n):=\frac 12 \sum_{i=1}^n \vert x_i -\bar{x}_i\vert$$
for $(x_1,\ldots,x_n)\in\R^n$, we have $\Phi_n^n(\bar{v},\cdot) = \delta_{\bar{x}}$, where $\delta_{\bar{x}}$ denotes the Dirac point measure concentrated at $\bar{x}$.
\end{lemma}

\goodbreak
We require the following result on Hessian measures of functions with lower dimensional domains. 

\begin{lemma}\label{extend}
Let $1\le k<n$. If $u\in\fconvs$ is such that $\dom u \subseteq E$ for some affine subspace $E\subset \R^n$ with $\dim E=k$, then
\begin{equation}
\label{eq:extend}
\int_{\R^n} \zeta(|y|^2)\d\Psi_j^n(u,y) = \int_{E} \int_{E^\perp} \zeta(\vert y_E\vert ^2+\vert y_{E^\perp}\vert^2) \d y_{E^\perp} \d\Psi_j^k(u\vert_{E},y_E)
\end{equation}
for every $1\le j\leq k$ and $\zeta\in C_c([0,\infty))$.
\end{lemma}
\begin{proof}
By the epi-translation invariance and rotation covariance of $\Psi_j^n$, we may assume without loss of generality that $\dom u \subset \R^k$. By \eqref{eq:int_u_psi_int_v_phi}, we have
$$\int_{\R^n} \zeta(|y|^2)\d\Psi_j^n(u,y)=\int_{\R^n} \zeta(|x|^2)\d\Phi_{j}^n(v,x),$$
where $v=u^*$ is in $\fconvf$. Since $\dom u\subseteq \R^k$, it follows that $v(x_1,\ldots,x_n)=w(x_1,\ldots,x_k)$ with $w\in\operatorname{Conv}(\R^k;\R)$ and therefore
$$\d\Phi_{j}^n(v,(x_1,\ldots,x_n))= \d\Phi_{j}^k(w,(x_1,\ldots,x_k))\d x_{k+1}\cdots \d x_n.$$
Thus, it follows from \eqref{eq:int_u_psi_int_v_phi} that
$$
\int\limits_{\R^n} \zeta(|y|^2)\d\Psi_j^n(u,y) = \int\limits_{\R^k} \int\limits_{\R^{n-k}} \zeta(|x_E|^2+|z|^2) \d z \d\Phi_{j}^k(w,x_E) =\int\limits_{\R^k} \int\limits_{\R^{n-k}} \zeta(|y_E|^2+|z|^2) \d z \d\Psi_j^k(u\vert_{\R^k},y_E)
$$
where we used that $w^*=u\vert_{\R^k}$, when taking the Legendre transform on $\R^k$.
\end{proof}

We remark that if we use \eqref{eq:extend} with $\zeta(t^2)=\xi(t)$
for $\xi\in C_c([0,\infty))$ and $t\geq 0$, then by \eqref{eq:abel_k}
\begin{equation}
\label{eq:extend_abel}
\int_{\R^n} \xi(|y|) \d\Psi_j^n(u,y) = \int_{E} \Abel^{n-k} \xi(|y_E|) \d\Psi_j^k(u\vert_{E},y_E)
\end{equation}
for every $u\in \fconvs$ and $1\leq j \leq k < n$ such that $\dom u \subseteq E$ for some affine subspace $E\subset \R^n$ with $\dim E=k$.

\subsection{Valuations on $\fconvs$}
\label{se:vals_on_convex_fcts}
We say that a valuation $\oZ\colon\fconvs\to\R$ is \emph{epi-homogeneous} of degree $j$ if $\oZ(\lambda\sq u)=\lambda^j\,\oZ(u)$ for every $\lambda>0$ and $u\in\fconvs$.
We recall the homogeneous decomposition theorem on the space $\fconvs$.

\begin{theorem}[\!\cite{Colesanti-Ludwig-Mussnig-4}, Theorem 1]
\label{thm:mcmullen_cvx_functions}
If $\,\oZ\colon\fconvs\to\R$ is a continuous and epi-translation invariant valuation, then there are continuous, epi-translation invariant valuations $\oZ_j\colon\fconvs\to\R$ for $j=0,\dots, n$ that are epi-homogeneous of degree $j$ such that $\oZ=\oZ_0+\cdots + \oZ_n$.
\end{theorem}

We say that a functional $\oZ\colon\fconvs\to\R$ is \emph{epi-additive} if
$$\oZ(u_1 \infconv u_2) =  \oZ(u_1)+\oZ(u_2)$$
for every  $u_1,u_2\in\fconvs$. 
The following result is a consequence of Theorem \ref{thm:mcmullen_cvx_functions}.

\begin{corollary}[\!\cite{Colesanti-Ludwig-Mussnig-4}, Corollary 22]\label{epi-additve}
If $\,\oZ\colon\fconvs\to \R$ is a continuous, epi-translation invariant valuation that is epi-homogeneous of degree $1$, then $\oZ$ is epi-additive.
\end{corollary}

We require the following classification results for valuations which are epi-homogeneous of degree $0$ and $n$.

\begin{theorem}[\!\cite{Colesanti-Ludwig-Mussnig-4}, Theorem 25]
\label{thm:class_0-hom}
A functional $\oZ\colon\fconvs \to \R$ is a continuous,  epi-translation invariant valuation that is epi-homogeneous of degree $0$ if and only if $\,\oZ$ is constant.
\end{theorem}

\begin{theorem}[\!\cite{Colesanti-Ludwig-Mussnig-4}, Theorem 2]
\label{thm:class_n-hom}
A functional $\oZ\colon\fconvs\to\R$ is a continuous, epi-translation invariant valuation that is epi-homogeneous of degree $n$ if and only if there exists $\zeta\in C_c(\R^n)$ such that
$$\oZ(u)=\int_{\dom u} \zeta(\nabla u(x))\d x$$
for every $u\in\fconvs$.
\end{theorem}

\goodbreak
As a simple consequence, we  obtain the following result (c.f. \cite[Corollary 2.5]{Colesanti-Ludwig-Mussnig-5}).

\begin{corollary}
\label{cor:class_n-hom-rot}
For $n\ge 2$, a functional $\oZ\colon\fconvs \to \R$ is a continuous, epi-translation and rotation invariant valuation that is epi-homogeneous of degree $n$ if and only if there exists $\zeta\in C_c([0,\infty))$ such that
$$\oZ(u)=\int_{\dom u} \zeta(|\nabla u(x)|) \d x$$
for every $u\in\fconvs$. For $n=1$, the same representation holds if we replace rotation invariance by reflection invariance.
\end{corollary}

\noindent
Here, we say that $\oZ$ is \emph{reflection invariant} if $\oZ(u)=\oZ(u^-)$, where $u^-(x)\deq u(-x)$ for $x\in\R$.

\goodbreak

By Theorem~\ref{thm:mcmullen_cvx_functions}, Theorem~\ref{thm:class_0-hom} and Corollary~\ref{cor:class_n-hom-rot} as well as the definition of the measures $\Psi_j^n$, we have the following characterization result for the $1$-dimensional case. See also \cite[Corollary 3]{Colesanti-Ludwig-Mussnig-4}.

\begin{proposition}
\label{prop:1}
A functional $\,\oZ:\fconvse \to \R$ is a continuous, epi-translation and reflection invariant valuation if and only if there exist $\zeta_0\in\Had{0}{1}$ and $\zeta_1\in \Had{1}{1}$ such that
$$\oZ(u)=\int_{\R} \zeta_0(\vert y\vert) \d \Psi^1_0(u, y)+\int_{\R} \zeta_1(\vert y\vert) \d \Psi^1_1(u, y)$$
for every $u\in\fconvse$. 
\end{proposition}

An important property of continuous, epi-translation invariant valuations on convex functions is that they are determined by their values on small sets of functions (cf.~\cite{Ludwig:SobVal, Colesanti-Ludwig-Mussnig-1}). We require the following version of this property that easily follows from \cite[Lemma 5.1]{Mussnig19}. 
For a convex body $K\in\cK^n$, we denote by $\ind_K\in\fconvs$ its (convex) indicator function, which is defined by
$$\ind_K(x):=\begin{cases}
0\quad &\text{for } x\in K,\\
+\infty\quad &\text{for } x\not\in K.\end{cases}$$
We denote by $\cP_{(0)}^n$ the set of polytopes in $\R^n$ that contain the origin in their interiors. Here, a polytope is the convex hull of finitely many points in $\R^n$.

\begin{lemma}
\label{le:reduction}
Let $\oZ_1,\oZ_2\colon\fconvs\to\R$ be continuous, epi-translation invariant valuations. If $\oZ_1(h_P+\ind_Q)=\oZ_2(h_P+\ind_Q)$ for every $P,Q\in\cP_{(0)}^n$, then $\oZ_1\equiv\oZ_2$ on $\fconvs$.
\end{lemma}

\subsection{Valuations on $\fconvf$}
For $X\subseteq\fconvx$, we associate with a valuation $\oZ\colon X\to\R$ its \emph{dual valuation} $\oZ^*$ defined on $X^*\deq\{u^*: u\in X\}$ by 
$$\oZ^*(u):=\oZ(u^*).$$
It was shown in \cite{Colesanti-Ludwig-Mussnig-3} that $\oZ\colon X\to\R$ is a continuous valuation if and only if
$\oZ^*\colon X^*\to\R$ is a continuous valuation. Since $u\in\fconvs$ if and only if $u^*\in\fconvf$, this allows us to transfer results between $\fconvs$ and $\fconvf$. We call a valuation $\oZ\colon \fconvf\to\R$ \emph{dually epi-translation invariant} if $\oZ^*$ is epi-translation invariant or equivalently
$$\oZ(v+\ell+\alpha)=\oZ(v)$$
for every $v\in\fconvf$, linear functional $\ell\colon\R^n\to\R$ and $\alpha\in\R$. We say that $\oZ$ is \emph{homogeneous} of degree $j$ if  $\oZ^*$ is epi-homogeneous of degree $j$ or equivalently
$$\oZ(\lambda \, v) = \lambda^j  \oZ(v)$$
for every $\lambda>0$ and $v\in\fconvf$.

The authors \cite{Colesanti-Ludwig-Mussnig-5} established the Hadwiger theorem also for valuations on $\fconvf$ by using duality with valuations on $\fconvs$. For $0\le j\le n$ and $\zeta\in\Had{j}{n}$, define $\oZZ{j}{\zeta}^{n,*}$ as the valuation dual to $\oZZ{j}{\zeta}^{n}$, that is, $\oZZ{j}{\zeta}^{n,*}(v):=\oZZ{j}{\zeta}^{n}(v^*)$ for $v\in\fconvf$.

\begin{theorem}[\!\cite{Colesanti-Ludwig-Mussnig-5}, Theorem 1.5]
\label{dthm:hadwiger_convex_functions}
A functional $\oZ\colon\fconvf \to \R$ is a continuous, dually epi-translation and rotation in\-va\-riant valuation if and only if there exist  functions $\zeta_0\in\Had{0}{n}$, \dots, $\zeta_n\in\Had{n}{n}$  such that
\begin{equation*}
\oZ(v)= \sum_{j=0}^n \,\oZZ{j}{\zeta_j}^{n,*}(v) 
\end{equation*}
for every $v\in\fconvf$.
\end{theorem}

\section{Simple and Dually Simple Valuations}
\label{se:simple_vals}
We define analogs of zonoids and generalized zonoids for convex functions and show that generalized functional zonoids are dense in $\fconvs$. This result is used to classify simple, continuous, and epi-additive functionals, and this classification result is used in the proof of Theorem \ref{Klain-Schneider}. 

A convex body $C\in\cK^{n+1}$ is called \emph{centrally symmetric} if some translate of $C$ is origin-symmetric. 
Note that centrally symmetric convex bodies with support function of class $C^\infty$ are dense in the set of all centrally symmetric convex bodies (see \cite[Theorem~3.4.1]{Schneider:CB2}).
A \emph{zonotope} is the Minkowski sum of finitely many line segments, and a convex body is a \emph{zonoid} if it is the limit of a sequence of zonotopes. 
A convex body $C$ is a \emph{generalized zonoid} if there exist zonoids $W,Z\in\cK^n$ such that
$$C+W=Z.$$
We require the following result (see, for example, \cite[Proposition 8.1.2]{Klain:Rota}).
\begin{lemma}
\label{generalized_zonoids}
If $C\in\cK^{n+1}$ is centrally symmetric with support function $h_C\in C^\infty(\S^n)$, then $C$ is a generalized zonoid.
\end{lemma}

We introduce functional analogs of zonoids and generalized zonoids as follows. Let $\fconvsz$ be the class of functions $u\in\fconvs$ that are the limit of a sequence $u_k$ where each $u_k\in\fconvs$ is the epi-sum of finitely many functions in $\fconvs$ with one-dimensional domains. Let $\fconvsg$ be the class of functions $u\in\fconvs$ such that there exist
$w,z\in\fconvsz$ with
$$u\infconv w=z.$$
We prove that $\fconvsg$ is dense in $\fconvs$.

\begin{lemma}
\label{convex_approx}
Let $n\ge2$. For $u\in\fconvs$, there exists a sequence $u_k\in\fconvsg$ such that  $u_k$ epi-converges to $u$ as $k\to \infty$.
\end{lemma}

\begin{proof}
Note that functions $u$ with $\dom u=\R^n$ are dense in $\fconvs$. So, fix  $u\in\fconvs$ with $\dom u=\R^n$ and consider 
$$u+ \ind_{k\Bn}$$
for $k\ge1$. Observe that this sequence is epi-convergent to $u$ as $k\to\infty$. Choose $\gamma_k>0$ so large that $u\le \gamma_k$ on $k\Bn$ and set 
$$K_k:=\epi(u+ \ind_{k\Bn}) \cap\{x\in\R^{n+1}\colon x_{n+1}\le \gamma_k\}.$$
Define
$$C_k:=K_k\cup (-K_k+2\gamma_ke_{n+1}).$$
and note that $C_k\in\cK^{n+1}$ is centrally symmetric. Let $C_{kl}\in\cK^{n+1}$ be a sequence of centrally symmetric convex bodies with $h_{C_{kl}}\in C^\infty(\S^n)$ converging to $C_k$ as $l\to\infty$. By Lemma \ref{generalized_zonoids}, for $k,l\ge 1$, there are zonoids $W_{kl},Z_{kl}\in\cK^{n+1}$ such that
\begin{equation}\label{zonoids}
C_{kl}+W_{kl}=Z_{kl}.
\end{equation}
Let $u_{kl}\in\fconvs$ be the lower envelope of $C_{kl}$, that is,
$$u_{kl}(x):=\sup\{t\in\R\colon (x,s)\not\in C_{kl} \text{ for } s\le t\}$$
for $x\in\R^n$. Note that $u_{kl}$ epi-converges to $u+\ind_{k\Bn}$ as $l\to \infty$. 
Let $w_{kl}, z_{kl}\in\fconvs$ be the lower envelopes of $W_{kl}$ and $Z_{kl}$, respectively. It follows from \eqref{zonoids} and the definition of epi-sum that
\begin{equation}\label{gz}
u_{kl}\infconv w_{kl} =z_{kl}.
\end{equation}
If $Z\in\cK^{n+1}$ is a zonotope, then the lower envelope of $Z$ is the epi-sum of finitely many convex functions with one-dimensional domains. Since $W_{kl}$ and $Z_{kl}$ are zonoids (and hence limits of zonotopes), we obtain that the functions $w_{kl}, z_{kl}\in\fconvsz$. It now follows from \eqref{gz} that $u_{kl}\in\fconvsg$. Hence there is a subsequence $l(k)$ such that $u_k\deq u_{k, l(k)}$ is a sequence in $\fconvsg$ that epi-converges to $u$.
\end{proof}

\goodbreak
We are now in the position to prove the following result.

\begin{proposition}
\label{prop:simple_epi-additive}
Let $\oZ\colon\fconvs \to \R$ be continuous and epi-additive. If $\,\oZ$ is simple and $n\ge 2$, then $\oZ$ vanishes identically.
\end{proposition}

\begin{proof}
If $w\in\fconvs$ is the epi-sum of finitely many functions in $\fconvs$ with lower dimensional domain, then the simplicity of $\oZ$ implies that
\begin{equation*}
\oZ(w)=0.
\end{equation*}
Since $\oZ$ is continuous, this implies that $\oZ$ vanishes on $\fconvsz$.

Let $u\in\fconvs$. By Lemma \ref{convex_approx}, we can approximate $u$ by a sequence
$u_k\in\fconvsg$, that is, there are sequences of convex functions $w_k, z_k\in\fconvsz$ such that
\begin{equation}\label{inf_eq}
u_k\infconv w_k= z_k
\end{equation}
and such that $u_k$ epi-converges to $u$.
Since $\oZ$ vanishes on $\fconvsz$, we obtain from \eqref{inf_eq} and the epi-additivity of $\oZ$ that 
$$\oZ(u_k\infconv w_k)= \oZ(u_k)=0$$
for $k\ge 1$.
Using that $\oZ$ is continuous on $\fconvs$ and that $u_{k}\in\fconvs$, we conclude that $\oZ$ vanishes identically. 
\end{proof}

For $n\geq 2$, a function $u\in\fconvs$ is called an \emph{orthogonal cylinder function} if it can be written as $u=u_E\infconv u_F$ where $u_E,u_F\in\fconvs$ are such that $\dom u_E\subseteq E$ and $\dom u_F\subseteq F$, where $E$ and $F$ are orthogonal and complementary subspaces with $\dim E,\dim F\geq 1$.

\begin{proposition}[\!\cite{Colesanti-Ludwig-Mussnig-5}, Proposition 4.6]
\label{prop:onehom}
Let $\,\oZ:\fconvs\to\R$ be a continuous and epi-translation invariant valuation. If $\,\oZ$ vanishes on all orthogonal cylinder functions and $n\geq 2$, then $\oZ$ is epi-homogeneous of degree $1$.
\end{proposition}

\subsection{Proof of Theorem \ref{Klain-Schneider}}
First, note that by Lemma \ref{gradient integral}, the functional $\oZ\colon\fconvs\to\R$, given by 
$$\oZ(u):=\int_{\dom u} \zeta(\nabla u(x)) \d x$$
with $\zeta\in C_c(\R^n)$, is a continuous, epi-translation invariant valuation, and that it is easy to see that $\oZ$ is simple.

We prove the converse statement by induction on the dimension $n$. For $n=1$, the statement follows from Theorem \ref{thm:mcmullen_cvx_functions}, Theorem \ref{thm:class_0-hom}, and Theorem \ref{thm:class_n-hom}. Let $n\geq 2$ and assume that the statement is true on $\fconvsk$ for $1\leq k \leq n -1$.

For $y\in\R^n$, set $\ell_y(x)\deq\langle x,y \rangle$ for $x\in\R^n$. By the properties of $\oZ$, it is easy to see that the map $K \mapsto \oZ(\ell_{y} +\ind_K)$ is a simple, continuous, translation invariant valuation on $\cK^n$ for every $y\in\R^n$. Hence, by Theorem~\ref{Klain}, there exists $\zeta(y)\in \R$ such that
\begin{equation}\label{KS}
\oZ(\ell_{y} + \ind_K) = \zeta (y)\, V_n(K)
\end{equation}
for every origin-symmetric $K\in\cK^n$. Note that since $\oZ$ is continuous, $\zeta$ depends continuously on $y$. 

\goodbreak
For $1\leq k\leq n-1$, let $E$ and $F$ be orthogonal subspaces of $\R^n$ such that $\R^n= E\times F$ where $\dim E=k$ and $\dim F=n-k$. 
For $x\in\R^n$, we write $x=(x_E,x_F)$ with $x_E\in E$ and $x_F\in F$. 
For $u_F\in \fconvsF$, the functional $u_E\mapsto \oZ(u_E\infconv u_F)$ is a simple, continuous, epi-translation invariant valuation on $\fconvsE$. 
Hence, by the induction hypothesis, there exists a function $\zeta_{u_F}\in C_c(E)$, depending on $u_F\in\fconvsF$, such that
\begin{equation}\label{uEuF}
\oZ(u_E\infconv u_F) = \int_{\dom u_E} \zeta_{u_F}(\nabla_E\, u_E(x_E)) \d x_E
\end{equation}
for every $u_E\in\fconvsE$. 
Setting $\bar u_E\deq\ell_{y_E}+\ind_{K_E}$ with $y_E\in E$ and $K_E\subset E$ an origin-symmetric convex body with positive $k$-dimensional volume, we obtain from (\ref{uEuF}) that 
$$\oZ(\bar u_E \infconv u_F)= V_k(K_E)\, \zeta_{u_F}(y_E).$$
Since $u_F\mapsto \oZ(\bar u_E\infconv u_F)$ is a simple, continuous,  epi-translation valuation on $\fconvsF$, it follows that also $u_F\mapsto \zeta_{u_F}(y_E)$ for $y_E\in E$ has these properties. From the induction hypothesis combined with \eqref{uEuF}, we therefore obtain that there exists $\zeta_{E,F}\in C_c(E\times F)$, depending on $E$ and $F$, such that
\begin{equation}
\label{eq:zeta_k}
\oZ(u_E \infconv u_F)=\int_{\dom u_E} \int_{\dom u_F} \zeta_{E,F}(\nabla_E\, u_E(x_E),\nabla_F\, u_F(x_F)) \d x_F \d x_E
\end{equation}
for every $u_E\in \fconvsE$ and $u_F\in\fconvsF$.
\goodbreak

Setting $u_E\deq\bar u_E$ and $u_F\deq\ell_{y_F}+\ind_{K_F}$ with  $y_F\in F$ and   $K_F\subset F$ an origin-symmetric convex body with positive $(n-k)$-dimensional volume, we obtain from \eqref{eq:zeta_k}  that
$$\oZ(\ell_{y_E+y_F} + \ind_{K_E\times K_F}) = \oZ((\ell_{y_E}+ \ind_{K_E})\infconv (\ell_{y_F} + \ind_{K_F}))= \zeta_{E,F}( y_E ,  y_F ) \,V_n(K_E\times K_F).$$
On the other hand, by \eqref{KS},
$$\oZ(\ell_{y_E+y_F} + \ind_{K_E\times K_F})=\zeta(y_E+y_F ) V_n(K_E\times K_F)$$
and therefore
$${\zeta}_{E,F}( y_E  ,   y_F  ) = \zeta(  y_E+y_F )$$
for every $y_E\in E$ and $y_F\in F$. It follows that $\zeta$ has compact support and that
$$\oZ(u) = \int_{\dom u} \zeta( \nabla u )\d x$$
for every $u=u_E\infconv u_F$ with $u_E\in\fconvsE$ and $u_F\in\fconvsF$. Note that this representation does not depend on the choice of $k$.

\goodbreak
Define $\oY\colon \fconvs \to \R$ by 
$$\oY(u) := \oZ(u) - \int_{\dom u} \zeta(\nabla u)\d x$$
for $u\in\fconvs$. It follows that $\oY$ is a simple, continuous, epi-translation invariant valuation that vanishes on all orthogonal cylinder functions. 
Proposition~\ref{prop:onehom} implies that $\oY$ is epi-homogeneous of degree~1. Hence, using Corollary \ref{epi-additve}, we obtain that $\oY$ is epi-additive. The statement now follows from Proposition \ref{prop:simple_epi-additive}.

\subsection{The Klain--Schneider Theorem on $\fconvf$}
A functional $\oZ\colon\fconvf\to \R$ is \emph{dually simple} if the dual functional $\oZ^*\colon\fconvs\to\R$ is simple. 
Equivalently, $\oZ$ is dually simple if $\oZ(v)=0$ for every $v\in\fconvf$ such that $v(x_E,x_{E^\perp})=w(x_E)$ for some $w\in\fconvfE$ where $E\subset \R^n$ is a linear subspace with $\dim E<n$ and where we write $x=(x_E,x_{E^\perp})$ for $x\in\R^n$ with $x_E\in E$ and $x_{E^\perp}\in E^\perp$.
This means that $\oZ$ vanishes on functions $v\in\fconvf$ which depend, in a suitable coordinate system, on less than $n$ variables.
We state the  following dual version of Theorem \ref{Klain-Schneider}, which is equivalent to the primal version by \eqref{n-gradient} and \eqref{eq:int_u_psi_int_v_phi}.

\begin{theorem}
A functional $\oZ\colon\fconvf \to \R$ is a dually simple, continuous, dually epi-translation invariant valuation if and only if there exists a function 
$\zeta\in C_c(\R^n)$  such that
\begin{equation*}
\oZ(v)= \int_{\R^n} \zeta(x) \d \Phi_n^n(v,x)
\end{equation*}
for every $v\in\fconvf$.
\end{theorem}

\section{Smooth Valuations}
\label{se:smooth_vals}
Let $\VConvf$ be the space of continuous and dually epi-translation invariant valuations on $\fconvf$.
Define $\cT\colon \VConvf \to \Val{n+1}$ by 
$$(\cT\circ\oZ)(K):=\oZ(h_K(\cdot, -1))$$
for $\oZ\in\VConvf$ and $K\in\cK^{n+1}$. 

Note that it is easy to see that the set 
$\{h_K(\cdot, -1)\colon K\in\cK^{n+1}\}$ is dense in $\fconvf$ (see \cite[Corollary 2.9]{Knoerr1} and \cite[Corollary 4.3.6]{Knoerr3}).
Also note that $\oZ$ is homogeneous of degree $j$ on $\fconvf$ if and only if $\cT\circ \oZ$ is homogeneous of degree $j$ on $\cK^{n+1}$, as
$$(\cT\circ\oZ)(\lambda K)= \oZ(h_{\lambda K}(\cdot,-1))=
\oZ(\lambda h_K(\cdot,-1))=\lambda^j\oZ(h_K(\cdot,-1))$$
for $\lambda\ge 0$ and $K\in\cK^{n+1}$.

Following \cite[Proposition 7.3.4]{Knoerr3}, we say that a continuous, dually epi-translation invariant valuation $\oZ\colon\fconvf\to\R$ is \emph{smooth} if 
$$\cT\circ\oZ\in\Vals{n+1}.$$
Since we restrict our attention to continuous valuations, in the following all smooth valuations are assumed to be also continuous.
Note that for a linear subspace $E$, the restriction of a smooth valuation $\oZ\colon\fconvf\to\R$ to $\fconvfE$  is again smooth.
A valuation $\oZ\colon\fconvs\to\R$ is called \emph{smooth} if its dual valuation $\oZ^*\colon\fconvf\to \R$ is smooth. 
As before, for a linear subspace $E$ of $\R^n$, the restriction of a smooth valuation $\oZ\colon\fconvs\to\R$ to $\fconvsE$  is again smooth.

We equip spaces of valuations with the topology of locally uniform convergence, which is equivalent to the compact-open topology. For a more detailed discussion of this topology, we refer to \cite[Section 6.1]{Knoerr1}.
Note that for valuations $\oZ_m, \oZ\colon \fconvs\to\R$, we have $\oZ_m\to\oZ$ if and only if for the dual valuations   $\oZ_m^*, \oZ^*\colon \fconvf\to\R$, we have  $\oZ_m^*\to \oZ^*$ (cf.\ \cite[Corollary 11.37]{RockafellarWets}). 

\goodbreak
Knoerr \cite{Knoerr3} established the following result for smooth valuations on $\fconvf$.

\begin{theorem}[\!\cite{Knoerr3}]
\label{hadwiger_smooth_dual}
A functional $\,\oZ\colon\fconvf \to \R$ is a smooth,  dually epi-translation and rotation invariant valuation if and only if there exist functions 
$\varphi_0,\dots, \varphi_n\in C_c^\infty([0,\infty))$  such that
\begin{equation*}
\oZ(v)=  \sum_{j=0}^n\int_{\R^n} \varphi_j(\vert x\vert^2) \d \Phi^n_j(v, x)
\end{equation*}
for every $v\in\fconvf$.
\end{theorem}

\noindent
The version stated in the introduction, Theorem \ref{hadwiger_smooth}, follows by duality, \eqref{eq:hess_def}, and \eqref{eq:int_u_psi_int_v_phi}. 

\goodbreak
It was shown in \cite{Knoerr2} that the set of smooth,  dually epi-translation invariant valuations on $\fconvf$ is dense in the set of continuous, dually epi-translation invariant valuations. We use the following dual version of Theorem 1 and Proposition 6.6 from \cite{Knoerr2}.

\begin{lemma}
\label{le:seq_smooth_vals}
If $\,\oZ\colon\fconvs\to\R$ is a continuous, epi-translation invariant valuation that is epi-homogeneous of degree $j$, then there exists a sequence of smooth, epi-translation invariant valuations $\oZ_m\colon\fconvs\to\R$, which are epi-homogeneous of degree $j$, that converges to $\oZ$.
If $\,\oZ$ is, in addition, rotation invariant, then also the valuations $\oZ_m$ can be chosen to have this property.
\end{lemma}

\goodbreak
We will also need the following result. Let $u_t(x)\deq t\, \vert x\vert + \ind_{\Bn}(x)$ for $t\ge 0$ and $x\in\R^n$.

\begin{lemma}
\label{le:smooth_vals_conv_unif_on_c}
If a sequence of valuations on $\fconvs$ converges to a valuation on $\fconvs$, then the sequence converges uniformly on $\{u_t\colon t\in C\}$ for every compact set $C\subseteq[0,\infty)$.
\end{lemma}
\begin{proof}
Since $u_t$ epi-converges to $u_{t_0}$ if and only if $t\to t_0$, it is easy to see that the compactness of $C$ implies that $\{u_t\colon t\in C\}$ is compact. The statement now follows  since the space of valuations is equipped with the topology of locally uniform convergence.
\end{proof}

Let $\S^n_{-}\deq$ $\{(z', z_{n+1}) \in \S^n \colon z_{n+1}<0\}$ be the negative half-sphere in $\R^{n+1}$, where $z'\deq(z_1, \dots, z_n)$. Recall that
using the central projection from $\S^{n}_{-}$ to the tangent plane at $-e_{n+1}$, which we identify with $\R^n$, we have
\begin{equation}\label{central}
\int_{\R^n} \xi(x)\d x= \int_{\S^n_{-}} \xi\Big(\frac{z'}{\vert z_{n+1}\vert}\Big)\, \vert z_{n+1}\vert^{n+1}\d\hm^n(z)
\end{equation}
for any continuous function $\xi\colon\R^n\to\R$ with compact support, where $\hm^n$ denotes the $n$-dimensional Hausdorff measure.

\goodbreak
\begin{lemma}\label{recalculate}
Let $\xi\in C_c(\R^n)$ and $1\leq j \leq n$.
If $\,\oZ: \fconvf\to\R$ is given by
\begin{equation*}
\oZ(v):= \int_{\R^n} \xi(x) \d\Phi^n_j(v,x),
\end{equation*}
then 
\begin{equation}\label{repres}
(\cT\circ \oZ)(K)=\int_{\S^n_{-}} \sigma(z) \d S_{j}(K,z)
\end{equation}
for every $K\in\cK^{n+1}$, where $\sigma \in C_c(\S^n_{-})$ is given by
\begin{equation}\label{calculation}
\sigma(z):={\binom{n}{j}}\, \xi\Big(\frac{z'}{\vert z_{n+1}\vert}\Big)\, \vert z_{n+1}\vert ^{n-j+1} 
\end{equation}
for $z=(z',z_{n+1})\in\S^n_{-}$.
\end{lemma}

\begin{proof} It suffices to prove (\ref{repres}) for $K\in\cK^{n+1}$ with smooth support function. In this case, we have
$$(\cT\circ \oZ)(K)=\oZ(h_K(\cdot,-1))=\int_{\R^n} \xi(x) \big[\Hess_{\R^n} h_K(x,-1)\big]_j\d x$$
where we write $\Hess_{\R^n}$ to stress that we take the Hessian of a function defined on $\R^n$.
As $h_K:\R^{n+1}\to\R$ is homogeneous of degree $1$, the $(n+1)\times (n+1)$ matrix $\Hess h_K$ is homogeneous of degree $(-1)$ and $\Hess h_K$ has at $z\in\R^{n+1}\backslash\{0\}$ an eigenvalue $0$ with corresponding eigenvector $z$.  Let $\lambda_k$ for $k=1, \dots, n$ be the remaining eigenvalues of $\Hess h_K(z)$. If $z=(x,-1)$, then the $n\times n$ matrix $\Hess_{\R^n} h_K(\cdot,-1)$ has at $x\in\R^n$ the eigenvalues $(1+|x|^2)\lambda_k $ for $k=1, \dots, n$. 
Indeed,  let $Q\subset\R^{n+2}$  be the  osculating  cylindrical para\-boloid of the graph of $h_K$ at $(z,h_K(z))$, let $E\subset \R^{n+2}$ be the hyperplane through $(z,0)$ that is orthogonal to $(z,0)$ and $F\subset\R^{n+2}$  the hyperplane through $(z,0)$ that is orthogonal to $(e_{n+1},0)$. 
The inter\-section of $Q$ with $E$ is the graph of a quadratic form with eigenvalues $\lambda_1, \dots, \lambda_n$ and the intersection of $Q$ with $F$ is the graph of a quadratic form given by $\Hess_{\R^n} h_K(x,-1)$. 
The result now follows by using the orthogonal projection from $F$ of $E$.
For $z=(x,-1)$, we now have
$$\big[\Hess_{\R^n} h_K(x,-1)\big]_j = \vert z\vert^{2j} \big[\Hess h_K(z)\big]_j.$$
Using the central projection and (\ref{central}), the above equation and the homogeneity (of degree $-1$) of $\Hess h_K$, we obtain
\begin{eqnarray*}
(\cT\circ \oZ)(K)&=&\int_{\S^n_{-}} \xi\Big(\frac{z'}{\vert z_{n+1}\vert}\Big) \Big[\Hess_{\R^n} h_K\Big(\frac{z'}{\vert z_{n+1}\vert},-1\Big)\Big]_j\, \vert z_{n+1}\vert^{n+1}\d \hm^n(z)\\
&=&\int_{\S^n_{-}} \xi\Big(\frac{z'}{\vert z_{n+1}\vert}\Big) 
\Big( 1+\frac{\vert z'\vert^2}{\vert z_{n+1}\vert^2}\Big)^j \,\Big[\Hess h_K\Big(\frac{z'}{\vert z_{n+1}\vert},-1\Big)\Big]_j\, \vert z_{n+1}\vert ^{n+1}\d \hm^n(z)\\
&=&\int_{\S^n_{-}} \xi\Big(\frac{z'}{\vert z_{n+1}\vert}\Big)  \big[\Hess h_K(z)\big]_j\, \vert z_{n+1}\vert ^{n-j+1}\d\hm^n(z).
\end{eqnarray*}
The statement now follows from the definition of the area measure $S_j(K,\cdot)$ (see \cite{Schneider:CB2}, Section 2.5).
\end{proof}

\goodbreak

This result implies the following lemma.

\begin{lemma}
Let $1\leq j \leq n$ and $\xi\in C_c(\R^n)$. The valuation $\oZ\colon \fconvf\to\R$,
defined by
$$\oZ(v):=\int_{\R^n} \xi(x) \d \Phi^n_j(v, x),$$
is smooth if and only if $\xi\in C_c^{\infty}(\R^n)$.
\end{lemma}

\begin{proof}
By definition, $\oZ$ is smooth if and only if $\cT\circ \oZ: \cK^{n+1}\to \R$ is smooth. By Lemma \ref{recalculate} and  \eqref{smooth_valuation}, this is the case if and only if $\sigma$ given in (\ref{calculation}) is smooth, which  is the case if and only if $\xi$ is smooth.
\end{proof}

By duality, we obtain a version of the previous Lemma on $\fconvs$.
\begin{lemma}
\label{le:smoothness_fconvs}
Let $1\leq j \leq n$ and $\xi\in C_c(\R^n)$. The valuation $\oZ\colon \fconvs\to\R$,
defined by
$$\oZ(u):=\int_{\R^n} \xi(x) \d \Psi^n_j(u, x),$$
is smooth if and only if $\xi\in C_c^{\infty}(\R^n)$.
\end{lemma}

\goodbreak
The following result is a consequence of Proposition \ref{prop:1} and Lemma \ref{le:smoothness_fconvs}.

\begin{proposition}
\label{prop:1 smooth}
A functional $\,\oZ\colon\fconvse \to \R$ is a smooth, epi-translation and reflection invariant valuation, if and only if there exist $\varphi_0, \varphi_1\in C_c^\infty([0,\infty))$ such that
$$\oZ(u)=  \int_{\R} \varphi_0(\vert y\vert^2) \d \Psi^1_0(u, y)+\int_{\R} \varphi_1(\vert y\vert^2) \d \Psi^1_1(u, y)$$
for every $u\in\fconvse$. 
\end{proposition}

\goodbreak
The next statement follows from Theorem \ref{Klain-Schneider} and Lemma~\ref{le:smoothness_fconvs}.

\begin{proposition}
\label{hadwiger_simple_smooth}
A functional $\,\oZ\colon\fconvs \to \R$ is a simple, smooth, epi-translation and rotation invariant valuation if and only if there exists a function 
$\varphi\in C_c^\infty([0,\infty))$  such that
\begin{equation*}
\oZ(u)= \int_{\R^n} \varphi(\vert y \vert^2) \d \Psi^{n}_n(u, y)
\end{equation*}
for every $u\in\fconvs$.
\end{proposition}

\subsection{Proof of Theorem \ref{hadwiger_smooth}}
First, note that for $\varphi_0,\dots, \varphi_n\in C_c^\infty([0,\infty))$, Lemma \ref{fine zeta} and Lemma~\ref{le:smoothness_fconvs} imply that the functional $\oZ\colon \fconvs\to\R$, defined by
\begin{equation*}
\oZ(u):= \sum_{j=0}^n\int_{\R^n} \varphi_j(\vert y\vert^2) \d \Psi^n_j(u, y),
\end{equation*}
is a smooth, epi-translation and rotation invariant valuation.

Conversely, let $\oZ\colon \fconvs\to \R$ be a smooth, epi-translation and rotation invariant valuation. 
For an $(n-1)$-dimensional subspace $E\subset \R^n$,  the restriction of $\oZ$ to $\fconvsE$ is a smooth, epi-translation and rotation invariant valuation on $\fconvsE$. If $\dim E=1$, then this restriction is also reflection invariant.
By induction on the dimension, starting with Proposition \ref{prop:1 smooth}, 
we obtain that for the restriction, $\oZ_E$, of $\oZ$ to $\fconvsE$ and $u\in\fconvsE$, 
$$\oZ_E(u)= \sum_{j=0}^{n-1} \int_{E} \bar \varphi_{j,E}(\vert y_E\vert^2) \d \Psi^{n-1}_j(u, y_E)$$
with suitable functions ${\bar\varphi_{j,E}}\in C_c^\infty([0,\infty))$. 
Since $\oZ$ is rotation invariant, the functions $\bar\varphi_{j,E}$ do not depend on $E$.
Thus there exist functions $\bar\varphi_j\in C_c^\infty([0,\infty)$ such that
$$\oZ(u)=\sum_{j=0}^{n-1} \int_E \bar\varphi_j(|y_E|^2) \d\Psi_j^{n-1}(u,y_E)$$
for every $(n-1)$-dimensional subspace $E\subset \R^n$ and every $u\in\fconvsE$. By Lemma~\ref{extend}, \eqref{eq:extend_abel} and the properties of the inverse Abel transform, there are functions $\varphi_j=\iAbel \bar\varphi_j\in C_c^{\infty}([0,\infty))$ such that
\begin{equation}\label{eq:ind_def}
\oZ(u)=\sum_{j=0}^{n-1} \int_{\R^n} \varphi_{j}(\vert y\vert^2) \d \Psi^{n}_j(u, y)
\end{equation}
for every $u\in\fconvs$ such that $\dom u$ is contained in some $(n-1)$-dimensional subspace of $\R^n$. The right side of \eqref{eq:ind_def} defines a smooth, epi-translation and rotation invariant valuation on $\fconvs$.

Hence,  the functional $\oZ_n\colon \fconvs\to \R$, defined by
$$\oZ_n(u):= \oZ(u)-  \sum_{j=0}^{n-1} \int_{\R^n} \varphi_j(\vert y\vert^2) \d \Psi^{n}_j(u, y),$$
is a smooth, epi-translation and rotation invariant valuation which is moreover simple. By Proposition~\ref{hadwiger_simple_smooth}, there exists a function $\varphi_n\in C_c^\infty([0,\infty))$ such that
$$\oZ_n(u)= \int_{\R^n} \varphi_n(\vert y \vert^2) \d \Psi^{n}_n(u, y)$$
for every $u\in \fconvs$. This completes the proof of the theorem.

\section{New Proof of Theorem \ref{thm:hadwiger_convex_functions}}
\label{se:proof_hadwiger}

We use the Cauchy--Kubota formulas for convex functions, which were recently established in \cite{Colesanti-Ludwig-Mussnig-6}, to  deduce the Hadwiger theorem for general valuations, Theorem \ref{thm:hadwiger_convex_functions}, from the Hadwiger theorem for smooth valuations, Theorem \ref{hadwiger_smooth}.

\goodbreak
Combining Lemma 3.4 from  \cite{Colesanti-Ludwig-Mussnig-6} and Corollary~\ref{cor:class_n-hom-rot}, we obtain the following result.

\begin{lemma}[\!\cite{Colesanti-Ludwig-Mussnig-6}]
\label{le:cauchy_kubota_is_continuous_val}
For $\,0\leq j \leq n$ and $\alpha\in C_c([0,\infty))$, the functional
\begin{equation*}
u\mapsto \int_{\Grass{j}{n}} \int_{\dom (\proj_E u)} \alpha(|\nabla_E\, \proj_E u(x_E)|) \d x_E \d E
\end{equation*}
is a continuous, epi-translation and rotation invariant valuation on $\fconvs$.
\end{lemma}

Recall that for $t\geq 0$ we define $u_t\in\fconvs$ as
$u_t(x)\deq t|x|+\ind_{\Bn}(x)$ for $x\in\R^n$. 

\begin{lemma}
\label{le:int_gjn_alpha_ut}
For $\alpha\in C_c([0,\infty))$ and $1\leq j \leq n-1$,
$$\int_{\Grass{j}{n}} \int_{\dom(\proj_E u_t)} \alpha(|\nabla_E\, \proj_E u_t(x_E)|) \d x_E \d E = \kappa_j\,  \alpha(t)$$
for every $t\geq 0$.
\end{lemma}
\begin{proof}
By \eqref{eq:proj_fct_lvl_set} it is easy to see that
$$\proj_E u_t(x_E)=t |x_E|+\ind_{B_E^j}(x_E)$$
for every $E\in \Grass{j}{n}$ and $x_E\in E$, where $B_E^j$ denotes the unit ball in the $j$-dimensional subspace $E$. Thus,
$$\int_{\Grass{j}{n}} \int_{\dom(\proj_E u_t)} \alpha(|\nabla_E\, \proj_E u_t(x_E)|) \d x_E \d E = \int_{\Grass{j}{n}} \int_{B_E^j} \alpha(t) \d x_E \d E = \kappa_j\, \alpha(t)$$
for every $t\geq 0$.
\end{proof}

The authors established in \cite{Colesanti-Ludwig-Mussnig-6} the following representation of functional intrinsic volumes. 

\begin{theorem}[\!\cite{Colesanti-Ludwig-Mussnig-6}, Theorem 1.6]
\label{cauchy_function}
Let $\,0\le j<n$. If $\zeta\in\Had{j}{n}$,  then
\begin{equation*}
\oZZ{j}{\zeta}^n(u)= \frac{\kappa_n}{\kappa_j\kappa_{n-j}} \binom{n}{j} \int_{\Grass{j}{n}}  \int_{\dom (\proj_E u)} \alpha(|\nabla_E\, \proj_E u(x_E)|) \d x_E \d E
\end{equation*}
for every $u\in\fconvs$, where $\alpha\in C_c([0,\infty))$ is given by 
$$\alpha(s):=  \kappa_{n-j} \big(  s^{n-j}\zeta(s)+(n-j)\int_s^{\infty}  t^{n-j-1} \zeta(t)\d t\big)$$
for $s>0$. 
\end{theorem}
\noindent
Here, for $j=0$, we set $\nabla_E\proj_E u(x_E)\deq0$.

\goodbreak
The following version of the Hadwiger theorem for convex functions was established in \cite{Colesanti-Ludwig-Mussnig-6}. Using Theorem \ref{cauchy_function} and properties of the integral transform which maps $\zeta$ to $\alpha$, the authors showed in \cite{Colesanti-Ludwig-Mussnig-6} that it is equi\-valent to Theorem~\ref{thm:hadwiger_convex_functions}. Let $n\ge 2$.

\begin{theorem}[\!\cite{Colesanti-Ludwig-Mussnig-6}, Theorem 1.7]
\label{thm2:hadwiger_convex_functions_ck}
A functional $\oZ:\fconvs \to \R$ is a continuous, epi-translation and rotation invariant valuation if and only if there exist  functions $\alpha_0, \dots, \alpha_n \in C_c([0,\infty))$ such that
\begin{equation}
\label{eq:hadwiger_convex_functions_ck}
\oZ(u)=  \sum_{j=0}^{n}\int_{\Grass{j}{n}}  \int_{\dom (\proj_E u)} \alpha_j(|\nabla_E\, \proj_E u(x_E)|) \d x_E \d E 
\end{equation}
for every $u\in\fconvs$. 
\end{theorem}

\subsection{New Proof of Theorem \ref{thm2:hadwiger_convex_functions_ck}}

For given $\alpha_0,\ldots,\alpha_n\in C_c([0,\infty))$, it follows from Lemma~\ref{le:cauchy_kubota_is_continuous_val} that the right side of \eqref{eq:hadwiger_convex_functions_ck} defines a continuous, epi-translation and rotation invariant valuation on $\fconvs$.

Conversely, let a continuous, epi-translation and rotation invariant valuation $\oZ\colon\fconvs\to\R$ be given. 
By Theorem \ref{thm:mcmullen_cvx_functions}, Theorem \ref{thm:class_0-hom} and Corollary \ref{cor:class_n-hom-rot}, it is enough to consider the case that $\oZ$ is epi-homogeneous of degree $j$ with $1\leq j \leq n-1$. 
By Lemma~\ref{le:seq_smooth_vals}, there exists a sequence of smooth, epi-translation and rotation invariant valuations $\oZ_m\colon\fconvs\to\R$, epi-homogeneous of degree $j$, that converges to $\oZ$ as $m\to \infty$. 
By Theorem~\ref{hadwiger_smooth} and Theorem~\ref{cauchy_function}, there exists $\alpha_m\in C_c^{\infty}([0,\infty))$ such that
$$\oZ_m(u)=\int_{\Grass{j}{n}} \int_{\dom(\proj_E u)} \alpha_m(|\nabla \proj_E u(x_E)|) \d x_E \d E$$
for every $u\in\fconvs$. 
It follows from Lemma~\ref{le:int_gjn_alpha_ut} that
\begin{equation}
\label{eq:z_i_alpha_i}
\oZ_m(u_t)=\kappa_j\, \alpha_m(t)
\end{equation}
for every $m\ge1$ and $t\geq 0$, where $u_t(x)=t|x|+\ind_{\Bn}(x)$
for $x\in\R^n$.
Define $\alpha\colon[0,\infty)\to \R$ by
$$\oZ(u_t)=: \kappa_j \alpha(t)$$ 
for $t\ge0$. By Lemma~\ref{le:smooth_vals_conv_unif_on_c}, the sequence $\oZ_m$ converges uniformly to $\oZ$ on $\{u_t\colon t \in C\}$ as $m\to \infty$ for any compact set $C\subset [0,\infty)$. 
Thus, also the sequence $\alpha_m$ converges uniformly to $\alpha$ on any compact subset of $[0,\infty)$, which implies that $\alpha\in C([0,\infty))$.

Next, we  show that $\alpha$ has compact support. Assume on the contrary that there exists a sequence $t_k\in[0,\infty)$ such that $t_k\to\infty$ as $k\to\infty$ and $\alpha(t_k)\neq 0$ for every $k\ge 1$. Without loss of generality, we may assume that $\alpha(t_k)>0$ for every $k\ge 1$. Define $w_k\in\fconvs$ as

$$w_k(x):= \Big(\frac1{\alpha(t_k)^{1/j}}\sq u_{t_k}\Big)(x)$$
for $x\in\R^n$. By (\ref{eq:z_i_alpha_i}) and the epi-homogeneity of $\oZ_m$,
$$\oZ_m(w_k)=\kappa_j \frac{\alpha_m(t_k)}{\alpha(t_k)}$$
for every $k, m\ge 1$. Observe that 
$$w_k(x)=t_k|x|+\ind_{(1/\alpha(t_k))^{1/j} \Bn}(x)$$
for $x\in\R^n$ and thus $w_k$ epi-converges to $\ind_{\{0\}}$ as $k\to\infty$. By the continuity of $\oZ_m$, we obtain that $\oZ_m(w_k)\to 0$ as $k\to\infty$ for every $m\ge 1$ since $\oZ_m(\ind_{\{0\}})=0$. Hence,
\begin{align*}
0=\lim_{m\to\infty} \oZ_m(\ind_{\{0\}})=\oZ(\ind_{\{0\}})=\lim_{k\to\infty} \oZ(w_k) = \lim_{k\to\infty} \lim_{m\to\infty} \oZ_m(w_k) = \lim_{k\to\infty} \lim_{m\to\infty} \kappa_j \frac{\alpha_m(t_k)}{\alpha(t_k)} = \kappa_j,
\end{align*}
which is a contradiction. Thus, we conclude that $\alpha\in C_c([0,\infty))$.

Finally, we show that the valuation $\oZ$ can be represented as in the statement of the theorem, that is,
\begin{equation}\label{representation}
\oZ(u)= \int_{\Grass{j}{n}}  \int_{\dom (\proj_E u)} \alpha(|\nabla_E\,\proj_E u(x_E)|) \d x_E \d E
\end{equation}
for $u\in\fconvs$. Lemma~\ref{le:cauchy_kubota_is_continuous_val} implies that the right side of \eqref{representation} defines a continuous, epi-translation and rotation invariant valuation on $\fconvs$. Hence, by Lemma~\ref{le:reduction} it suffices to show that $\oZ$ and the right side of \eqref{representation} coincide on functions of the form $h_P+\ind_Q$ with $P,Q\in\cP_{(0)}^n$. To show this, fix polytopes $P,Q\in\cP_{(0)}^n$. Note that $\dom \proj_E(h_P+\ind_Q)=\proj_E Q$ for $E\in\Grass{j}{n}$.
Since support functions are Lipschitz, there exists $t_0\ge 0$ such that $\vert y\vert \le t_0$ for every $y\in\partial(h_P+\ind_Q)(x)$ for all $x\in\interno Q$, where $\interno$ stands for interior.  Lemma~\ref{le:proj_subgradient} implies that $\vert y_E\vert \le t_0$ for every $E\in\Grass{j}{n}$ and $y_E\in \partial \proj_E(h_P+\ind_Q)(x_E)$ for all $x_E\in \proj_E\interno Q$. Hence
\begin{equation}
\label{eq:nabla_proj_bd}
|\nabla_E\, \proj_E(h_P+\ind_Q)(x_E)|\leq t_0
\end{equation}
for every $E\in\Grass{j}{n}$ and a.e.\! $x_E\in\proj_E Q$.
Since $\alpha_m$ converges uniformly to $\alpha$ on $[0,t_0]$,
it follows from \eqref{eq:nabla_proj_bd} that
$$
\lim_{m\to\infty} \int_{\proj_E Q} \alpha_m(|\nabla_E \proj_E(h_P+\ind_Q)(x_E)|) \d x_E = \int_{\proj_E Q} \alpha(|\nabla_E \proj_E(h_P+\ind_Q)(x_E)|) \d x_E
$$
for  $E\in \Grass{j}{n}$. Since $\alpha_m$ is uniformly bounded on $[0,t_0]$, the space $\Grass{j}{n}$ is compact and the function $E\mapsto V_j(\proj_E Q)$ is continuous, there is $\gamma>0$ such that
$$\Big \vert\int_{\proj_E Q} \alpha_{m}(|\nabla_E\, \proj_E(h_P+\ind_Q)(x_E)|) \d x_E \Big\vert \leq \gamma \max_{E\in \Grass{j}{n}} V_j(\proj_E Q)$$
for every $E\in \Grass{j}{n}$. Hence, we may apply the dominated convergence theorem 
to obtain that
\begin{align*}
\oZ(h_P+\ind_Q)&=\lim_{m\to\infty} \oZ_m(h_P+\ind_Q)\\
&=\lim_{m\to\infty} \int_{\Grass{j}{n}} \int_{\proj_E Q} \alpha_m(|\nabla_E\, \proj_E(h_P+\ind_Q)(x_E)|) \d x_E \d E\\
&= \int_{\Grass{j}{n}} \int_{\proj_E Q} \alpha(|\nabla_E\, \proj_E(h_P+\ind_Q)(x_E)|) \d x_E \d E,
\end{align*}
which concludes the proof.

\section{Functions with Lower Dimensional Domain and Klain's Proof} 
\label{se:lower}

We discuss the extension of valuations on functions with lower dimensional domain to valuations on general functions. We show that the Abel transform plays a critical role and explain why Klain's approach \cite{Klain95} to Hadwiger's theorem does not work in general in the functional setting.

\subsection{Further Properties of the Abel Transform}
The following lemma and Proposition~\ref{prop:hessian_vals_on_lower_dim_functions} show that the Abel transform, which was defined in Section~\ref{se:abel}, maps $\Had{j}{n}$ to $\Had{j}{n-1}$.

\begin{lemma}
\label{le:abel_hadwiger_class}
Let $\,0\leq j \leq n-2$. If $\zeta\in\Had{j}{n}$, then $\Abel \zeta\in \Had{j}{n-1}$.
\end{lemma}
\begin{proof}
It is easy to see that the condition $\zeta\in C_b((0,\infty))$ implies that also $\Abel \zeta\in C_b((0,\infty))$. Since $\zeta\in\Had{j}{n}$, there exist $\beta\ge 0$ such that
\begin{equation}
\label{eq:int_zeta_moment_bd}
\Big\vert\int_s^{\infty}t^{n-j-1}\zeta(t)\d t\Big\vert \le \beta
\end{equation}
for every $s>0$ and  $\gamma\deq\lim_{s\to 0^+}\int_s^{\infty}t^{n-j-1}\zeta(t)\d t$ exists and is finite. Using \eqref{eq:def_abel_trans2} and the substitution $z=t\cosh(r)$, we obtain
\begin{align*}
\lim_{s\to 0^+} \int_s^{\infty} t^{n-j-2}\Abel \zeta(t)  \d t&= 2 \lim_{s\to 0^+} \int_s^{\infty} \int_{0}^{\infty} t^{n-j-1}\zeta(t \cosh(r)) \cosh(r)  \d r \d t\\
&=2 \lim_{s\to 0^+} \int_0^{\infty} \frac{1}{\cosh^{n-j-1}(r)} \int_{s \cosh(r)}^{\infty} z^{n-j-1}\zeta(z)  \d z \d r\\
&= 2 \int_0^{\infty} \frac{\gamma}{\cosh^{n-j-1}(r)} \d r
\end{align*}
where we have used \eqref{eq:int_zeta_moment_bd} and the dominated convergence theorem in the last step. Using that
\begin{equation}
\label{eq:cosh}
\int_0^{\infty} \frac1{\cosh^{n-j-1}(r)} \d r \le \int_0^{\infty} \frac1{\cosh(r)} \d r =\frac{\pi}2,
\end{equation}
we obtain that $\lim_{s\to 0^+} \int_s^{\infty} t^{n-j-2} \Abel \zeta(t) \d t$ exists and is finite. Note that for $j=0$ we have thus shown that $\Abel \zeta\in\Had{0}{n-1}$.

\goodbreak
Next, let  $1\leq j \leq n-2$ and $\varepsilon>0$. Since $\zeta\in\Had{j}{n}$, there exists $\delta>0$ such that $|r^{n-j} \zeta(r)|<\varepsilon$ for every $r\in(0,\delta)$. Thus,
\begin{align*}
\Big|s^{n-j} \int_0^{\acosh(\delta/s)} \cosh(r) \zeta(s\cosh(r))\d r \Big| &\leq  \int_0^{\acosh(\delta/s)} \big|(s \cosh(r))^{n-j} \zeta(s \cosh(r))\big| \frac{1}{\cosh^{n-j-1}(r)} \d r\\
&\leq \varepsilon \int_0^{\infty} \frac{1}{\cosh^{n-j-1}(r)} \d r.
\end{align*}
Therefore
\begin{align*}
\left|s^{n-j-1} \Abel\zeta(s)\right| &= \Big|2 s^{n-j} \int_0^{\infty} \cosh(r) \zeta(s \cosh(r)) \d r\Big|\\
&\leq 2 \varepsilon \int_0^{\infty} \frac{1}{\cosh^{n-j-1}(r)} \d r + 2 s^{n-j} \int_{\acosh(\delta/s)}^{\infty} \cosh(r) |\zeta(s \cosh(r))| \d r\\
&=2 \varepsilon \int_0^{\infty} \frac{1}{\cosh^{n-j-1}(r)} \d r + 2 s^{n-j-1} \int_{\delta}^{\infty} \big|\zeta(\sqrt{s^2+t^2})\big| \d t.
\end{align*}
Since $\zeta$ is continuous with bounded support and $\delta>0$, the last inequality implies that 
$$\limsup_{s\to 0^+} |s^{n-j-1} \Abel \zeta(s)| \leq 2 \varepsilon \int_0^{\infty} \frac1{\cosh^{n-j-1}(r)} \d r.$$ 
Using \eqref{eq:cosh} again,
we obtain  $\lim_{s\to 0^+} s^{n-j-1} \Abel \zeta(s) = 0$ and have shown that $\Abel \zeta\in \Had{j}{n-1}$.
\end{proof}

We will see in Proposition~\ref{prop:hessian_vals_on_lower_dim_functions} that the  statement from Lemma \ref{le:abel_hadwiger_class} is also true for $j=n-1$.

\subsection{Functions with Lower Dimensional Domains}
Let $\fconvfs$ denote the set of functions in $\fconvf$ that are of class $C^2$ in a neighborhood of the origin. It is easy to see that $\fconvfs$ is dense in $\fconvf$.
We need the following three lemmas.

\begin{lemma}[\!\cite{Colesanti-Ludwig-Mussnig-5}, Lemma 3.1]\label{lem:existence_v_smooth_origin} Let $\,1\leq j \leq n$ and $\zeta\in \Had{j}{n}$. If $\,v\in\fconvfs$, then
\begin{equation*}
\int_{\R^n}\big\vert\zeta(|x|)\big\vert \d\Phi^n_{j}(v,x)
\end{equation*}
is well-defined and finite.
\end{lemma}

\begin{lemma}[\!\cite{Colesanti-Ludwig-Mussnig-5}, Lemma 3.23]
\label{le:dual_hessian_vals_on_fconvfs}
If $\,0\leq j \leq n$ and $\zeta\in\Had{j}{n}$, then
$$\oZZ{j}{\zeta}^{n,*}(v)=\int_{\R^n} \zeta(|x|) \d\Phi_j^n(v,x)$$
for every $v\in\fconvfs$.
\end{lemma}

\goodbreak
\begin{lemma}
\label{le:lower_dim_fconvfs}
Let $1\leq k <n$ and $v\in\fconvfs$. If   there exists $w\in\fconvfsk$ with 
$$v(x_1,\ldots,x_n)=w(x_1,\ldots,x_k)$$ 
for every $(x_1,\dots, x_n)\in\R^n$, then
$$\int_{\R^n} \zeta(|x|) \d\Phi_j^n(v,x) = \int_{\R^k} \Abel^{n-k}\zeta(|z|)  \d\Phi_{j}^k(w,z)$$
for every $0\leq j \leq k$ and $\zeta\in\Had{j}{n}$.
\end{lemma}
\begin{proof}
By our assumptions on $v$ and $w$, it follows from Lemma~\ref{le:decompose_hessian_measure} that
$$\d\Phi_j^n(v,(x_1,\ldots,x_n))=\d\Phi_j^k(w,(x_1,\ldots,x_k)) \d x_{k+1}\cdots \d x_n.$$
By Lemma~\ref{lem:existence_v_smooth_origin} the function $x\mapsto \zeta(|x|)$ is absolutely integrable with respect to $\d\Phi_j^n(v,\cdot)$. Thus, 
$$\int_{\R^n} \zeta(|x|) \d\Phi_j^n(v,x)= \int_{\R^k} \int_{\R^{n-k}} \zeta(\sqrt{|y|^2+|z|^2}) \d z \d\Phi_j^k(w,z)= \int_{\R^k} \Abel^{n-k}\zeta(|z|)  \d\Phi_{j}^k(w,z)$$
by Fubini's theorem.
\end{proof}

The following result describes the behavior of functional intrinsic volumes on functions with lower dimensional domain. We require the following definition. For a given $k$-dimensional affine subspace $E\subset \R^n$ there exist a translation $\tau$ on $\R^n$ and $\vartheta\in\SO(n)$ such that $\{\vartheta(\tau x)\colon x\in E\}=\R^k$, where we also consider $\R^k$ as a subspace of $\R^n$. Note that $\tau$ and $\vartheta$ are not unique. Now for every $u\in\fconvs$ with $\dom u \subseteq E$ we have $\dom (u\circ\tau^{-1}\circ\vartheta^{-1})\subseteq \R^k$. 
Hence, we may  consider $u\circ\tau^{-1}\circ\vartheta^{-1}$ as an element of $\fconvsk$. If $\oZ\colon\fconvsk\to\R$ is epi-translation and $\O(k)$ invariant, where $\O(k)$ denotes the orthogonal group in $\R^k$, we set
$$\oZ(u):=\oZ(u\circ \tau^{-1}\circ \vartheta^{-1})$$
for $u\in\fconvs$ with $\dom u\subseteq E$. Since $\oZ$ is epi-translation and $\O(k)$ invariant, this definition does not depend on the particular choice of the translation $\tau$ and $\vartheta\in\SO(n)$ and thus, $\oZ$ is well defined on $\{u\in\fconvs\colon \dom u \subseteq E\}$. Note that it is easy to see that functional intrinsic volumes on $\fconvsk$ are not only rotation invariant but even $\O(k)$ invariant.

\begin{proposition}
\label{prop:hessian_vals_on_lower_dim_functions}
Let $1\le k<n$. If $u\in\fconvs$ is such that $\dom u \subseteq E$ for some affine subspace $E\subset \R^n$ with $\dim E=k$,  then
$$\oZZ{j}{\zeta}^n(u)=\oZZ{j}{\Abel^{n-k}\zeta}^k(u|_E)$$
for every $0\leq j\leq k$ and $\zeta\in\Had{j}{n}$. In particular, $\Abel^{n-k}\zeta \in \Had{j}{k}$.
\end{proposition}

\noindent
We deduce this proposition from its dual version, which is stated below, using the following facts and definitions. If $u\in\fconvs$ is such that $\dom u \subseteq \R^k$ and $v\in \fconvf$ is its Legendre transform, then  there exists $w\in\fconvfk$ such that $v(x_1,\ldots,x_n)=w(x_1,\ldots,x_k)$ for $(x_1, \dots, x_n)\in\R^n$. Moreover, $w^*=u$ on $\R^k$, when taking the Legendre transform on $\R^k$. Similar as above, given an $\O(k)$ invariant $\oZ:\fconvfk\to\R$ and $E\in\Grass{k}{n}$, we set
$$\oZ(v):=\oZ(v\circ \vartheta^{-1})$$
for $v\in\fconvfE$ with a suitable $\vartheta\in\SO(n)$.

\begin{proposition}
For $1\le k <n$, let $E\in\Grass{k}{n}$ and write $x=(x_E,x_{E^\perp})$ for $x\in \R^n$ with $x_E\in E$ and $x_{E^\perp}\in E^\perp$. If $v\in\fconvf$ is such that $v(x_E,x_{E^\perp})=w(x_E)$ for some $w\in \fconvfE$, then
$$\oZZ{j}{\zeta}^{n,*}(v)=\oZZ{j}{\Abel^{n-k} \zeta}^{k,*}(w)$$
for every $0\leq j \leq k$ and $\zeta\in\Had{j}{n}$. In particular, $\Abel^{n-k}\zeta \in \Had{j}{k}$.
\end{proposition}

\begin{proof}
Without loss of generality, we may assume that $E=\R^k$. By Lemma~\ref{le:lower_dim_fconvfs} and Lemma~\ref{le:dual_hessian_vals_on_fconvfs}, we see that the statement holds if $v\in\fconvfs$ (and therefore $w\in\fconvfsk$). Since $\fconvfs$ is dense in $\fconvf$ and $\fconvfsk$ is dense in $\fconvfk$ while $\oZZ{j}{\zeta}^{n,*}$ is continuous on $\fconvf$, it remains to show that $\oZZ{j}{\Abel^{n-k}\zeta}^{k,*}$ is well-defined and continuous on $\fconvfk$. By Theorem~\ref{dthm:hadwiger_convex_functions}, it suffices to show that $\Abel^{n-k}\zeta\in\Had{j}{k}$.

If $j<k$, then the statement follows from Lemma~\ref{le:abel_hadwiger_class}. In the remaining case $j=k$, we need to show that $\Abel^{n-k}\zeta\in\Had{k}{k}$. Since $\zeta$ is continuous with bounded support, it is easy to see that also $\Abel^{n-k}\zeta$ is continuous with bounded support. It remains to show that $\lim_{s\to 0^+} \Abel^{n-k} \zeta(s)$ exists and is finite.
Let $\bar{x}_1,\ldots,\bar{x}_k\in\R\backslash \{0\}$ be such that $|\bar{x}_E|=1$, where $\bar{x}_E\deq(\bar{x}_{1},\ldots,\bar{x}_{k})$. For $t\geq 0$, define $w_t\in\fconvfk$ as
$$w_t(x_1,\ldots,x_k):=\frac 12 \sum_{i=1}^k \vert x_i- t \bar{x}_{i} \vert$$
for $(x_1,\ldots,x_k)\in\R^k$ and let $v_t\in\fconvf$ be such that $v_t(x_1,\ldots,x_n)\deq w_t(x_1,\ldots,x_k)$ for $(x_1,\ldots,x_n)\in\R^n$. Note that $w_t\in\fconvfsk$ and $v_t\in\fconvfs$ if $t>0$. By Lemma \ref{le:dual_hessian_vals_on_fconvfs}, Lemma~\ref{le:lower_dim_fconvfs} and Lemma~\ref{le:delta} we now have
$$\oZZ{k}{\zeta}^{n,*}(v_t)=\int_{\R^n} \zeta(|x|)\d\Phi_k^n(v_t,x)=\int_{\R^k} \Abel^{n-k} \zeta(|x_E|) \d\Phi_k^k(w_t,x_E)=\Abel^{n-k}\zeta(t)$$
for every $t>0$. Since $\zeta\in\Had{k}{n}$, the functional intrinsic volume $\oZZ{k}{\zeta}^{n,*}$ is continuous on $\fconvf$, and 
$$\lim_{t\to 0^+}\oZZ{k}{\zeta}^{n,*}(v_t)= \oZZ{k}{\zeta}^{n,*}(v_0).$$
Thus, $\lim_{t\to 0^+}\Abel^{n-k} \zeta(t)$ exists and is finite, which completes our proof.
\end{proof}

\goodbreak
\subsection{On Klain's Proof}
In our proof of the Hadwiger theorem for smooth valuations, Theorem \ref{hadwiger_smooth}, we adapted Klain's approach \cite{Klain95} to the Hadwiger theorem. We used induction on the dimension and extended valuations on functions with lower dimensional domains to valuations on functions with general domains. This step required to invert the Abel transform, which is possible in the setting of smooth valuations.

Let $1\le j\le k$ and $\zeta\in\Had{j}{k}$. If $\zeta\not\in C^1((0,\infty))$, then by Lemma~\ref{le:abel_sq_diffable} and Proposition~\ref{prop:hessian_vals_on_lower_dim_functions} there does not exist any $\xi\in\Had{j}{k+2}$ such that
$$\oZZ{j}{\xi}^{k+2}(u)=\oZZ{j}{\zeta}^{k}(u|_{\R^k})$$
for every $u\in\fconvskpt$ with $\dom u\subseteq \R^k$. 
More generally, if $\zeta\in\Had{j}{k}$ is not in the image of the Abel transform (as map from $\Had{j}{k+1}$ to $\Had{j}{k}$), then $\oZZ{j}{\zeta}^k$ cannot be obtained by restricting any functional intrinsic volume of the form $\oZZ{j}{\xi}^{k+l}$, with $l>0$, to functions with $k$-dimensional domain. 
This also means that it is not possible to extend $\oZZ{j}{\zeta}^k$ to functions defined on any higher-dimensional space.

In other words, the dimension of the ambient space matters, which is in contrast to classical intrinsic volumes. This is also the reason why we used Klain's approach only for smooth valuations. 

\subsection*{Acknowledgments}
The authors thank Semyon Alesker, Daniel Hug, Jonas Knoerr, and Boris Rubin for their helpful remarks.
 M.~Ludwig was supported, in part, by the Austrian Science Fund (FWF):  P~34446 and F.~Mussnig was supported by the Austrian Science Fund (FWF): J 4490-N.

\end{document}